%% file: main.tex
\documentclass[reqno]{amsart}
\input{macro.tex}

\title[The cone conjecture for primitive symplectic varieties over a field of characteristic 0]{The cone conjecture for primitive symplectic varieties over a field of characteristic zero and an application}
\author{Aurélien Faucher}
\date{\today}

\begin{document}

\begin{abstract}
We prove the Kawamata-Morrison cone conjecture for $\Q$-factorial terminal projective primitive symplectic varieties with second Betti number greater than five defined over a field of characteristic $0$. As an application, we prove that the relative movable and the relative nef cone conjectures hold for fibrations whose very general fibre is a projective primitive symplectic varieties under certain assumptions.
\end{abstract}

\maketitle

\section{Introduction}
For a smooth projective variety $X$ over a field $F$, it is known thanks to the famous Cone Theorem that the $K_X$-negative part of the cone of curves of $X$ has a rather simple description: it is locally rational polyhedral, and when the characteristic of $F$ is equal to $0$, the extremal $K_X$-negative rays are associated with contraction morphisms \cite[Theorem 3.7]{Kollár_Mori_1998}. Nevertheless, the form of the $K_X$-positive part is generally not as predictable. For example, an abelian surface with Picard number at least $3$ has a round nef cone, or a K$3$ surface may well possess an infinite number of $(-2)$-curves (see for example \cite{totaro2012surface}). All hope is not lost regarding the study of this last part: the Kawamata-Morrison cone conjecture establishes the existence of a rational polyhedral fundamental domain under the action of certain groups of transformations. The statement of this conjecture \cite{morrison96beyond} has been greatly generalised over time, see for example \cite[Conjecture 1.12]{kawamata97cone} and \cite[Conjecture 1.1]{totaro10conepairs}. In this paper, we will state a version without boundaries. To do this, we introduce a general definition (the various notations used will be clearly defined in Section \ref{sect_relative_prelim}).

\begin{defn}\label{def_ktrivfibrespace}
A $K$-trivial fiber space is a normal $\Q$-factorial klt pair $(X,\Delta)$ endowed with a fibration $f : X \to S$ (i.e. a proper surjective morphism with connected fibres) over a quasi-projective variety $S$ such that $K_X + \Delta \equiv_f 0$. 
\end{defn}

\begin{conj}\label{conj_relKM}
Let $f : X \to S$ be a $K$-trivial fiber space.

\begin{enumerate}
\item There exists a rational polyhedral cone $\Pi$ which is a fundamental domain under the action of the group of relative pseudoautomorphisms $\psaut(X/S)$ on $\clMov^e(X/S) := \clMov(X/S) \cap \Eff(X/S)$ in the following sense:
\begin{itemize}
    \item $\clMov^e(X/S) = \cup_{g \in \psaut(X/S)} g^*\Pi$,
    \item $\forall g \in \psaut(X/S): \Pi^\circ \cap (g^*\Pi)^\circ = \emptyset$, except for $g^* = 1$ seen as an element of $\GL(N^1(X/S))$.
\end{itemize}
\item There exists a rational polyhedral cone $\Pi'$ which is a fundamental domain under the action of the group of relative automorphisms $\Aut(X/S)$ on $\Nef^e(X/S) := \Nef(X/S) \cap \Eff(X/S)$ in the above sense.
\end{enumerate}
\end{conj}

\begin{defn}
	With the notations introduced above, item $(1)$ and item $(2)$ will be, respectively, referred to as \emph{the relative movable cone conjecture} and \emph{the relative nef cone conjecture} for $X$ over $S$. When $S$ is the spectrum of a field, thoses items will simply be referred to as the \emph{absolute movable cone conjecture} and the \emph{the absolute nef cone conjecture}.
\end{defn}

Beyond its intrinsic interest, the Kawamata–Morrison cone conjecture has important implications for the birational geometry of varieties with trivial canonical divisor, see \cite[Theorem 1.5]{gachet2024effective}. In the relative setting, we can mention the work of Y. Kawamata in dimension three \cite{kawamata97cone}, as well as the recent works \cite{li2025relative}, \cite{li2023relative} and \cite{moraga2024geometric}, which establish weak versions of the relative cone conjecture for fibrations in surface.

The absolute cone conjecture has been verified in a number of cases, but remains open in dimension at least three. For instance, the case of smooth surfaces is covered by the work of H. Sterk \cite{sterk1985finiteness} (inspired by ideas of E. Looijenga), Y. Namikawa \cite{Namikawa1985} and Y. Kawamata \cite{kawamata97cone}. The conjecture was proved for klt surface pairs by B. Totaro in \cite{totaro10conepairs}. A.Prendergast-Smith established the conjecture for abelian varieties over an algebraically closed field in \cite{Prendergast-Smith2012}. Thanks to the work of numerous geometers, this conjecture has been proven for irreducible holomorphic symplectic manifolds $X$ when the effective nef cone $\Nef^e(X)$ is replaced by the convex hull of $\Nef(X) \cap N^1(X)_\Q$ which is usually denoted by $\Nef^+(X)$, see for example \cite{amerik2017morrison}, \cite{amerik2020collections}, \cite{markman2011survey} and \cite{markman2015proof}. In general, Boucksom-Zariski divisorial decomposition allows us to show that $\Nef^e(X) \subset \Nef^+(X)$, but the reverse inclusion is more difficult to prove when we do not know whether integral nef isotropic classes (with respect to the Beauville-Bogomolov-Fujiki form) are effective. However, given that the SYZ conjecture for irreducible holomorphic symplectic manifolds \cite[Conjecture 4.1]{sawon2003abelian} has been proven for the four known families (\cite[Theorem 1.5]{bayer2014mmp}, \cite[Theorem 1.3]{markman2014lagrangian}, \cite[Proposition 3.38]{yoshioka16bridgeland}), \cite[Corollary 1.1]{matsushita17isotropic} \cite[Theorem 7.2, Corollary 7.3]{mongardi2021monodromy} and \cite[Theorem 2.2]{mongardi2022birational}) we have in these cases the equality $\Nef^e(X) = \Nef^+(X)$ and consequently, the Kawamata-Morrison conjectures are proven for these families. The relative movable cone conjecture has been proven in \cite{horing2024relative} for a fibration whose very general fibre is a projective irreducible holomorphic symplectic manifold, assuming that the total space is klt and assuming that every nef divisor on the very general fibre is semiample. 

\vspace{0.1cm}

The absolute cone conjecture has also been established for complex projective primitive symplectic varieties (see Definition \ref{def_PSV}), which are the singular analogues of complex projective irreducible holomorphic symplectic manifolds, under suitable assumptions in \cite[Theorem~1.2]{lehn2024morrison}. These varieties are equally interesting given that they form one of the three classes along with complex tori and singular Calabi-Yau varieties that can be used to express, up to a quasi-étale covering, any varieties with log-terminal singularities and trivial first Chern class, see for example \cite{druelguenancia2018decomposition}, \cite{druel2018decomposition}, \cite{greb2019klt}, \cite{horing2019algebraic} and \cite{bakker2022algebraic}. Such varieties may appear when considering certain moduli spaces; for example, if $S$ is a K3 surface or if $A$ is an abelian surface, for a fixed Mukai vector $v$ on $S$ or $A$, the spaces $M_v(S,H)$ and $K_v(A,H)$ are almost always primitive symplectic varieties with terminal singularities (see for instance \cite[Theorem 1.10]{perego2023irreducible}). We can also mention Nikulin orbifolds: starting from an irreducible holomorphic symplectic variety $X$ of type $K3^{[2]}$ with a symplectic involution $\iota \in \Aut(X)$, the quotient $\hat{X} = X / \iota$ is singular along 28 isolated points and a K3 surface denoted $\Sigma$. The blow-up of $\hat{X}$ along $\Sigma$ is then a primitive symplectic variety called a Nikulin orbifold, see \cite[Proposition 3.8]{menet2020global}. The first result of this paper is the proof of the absolute nef and movable cone conjectures for projective primitive symplectic varieties over a field of characteristic $0$.

\begin{thmx}\label{thm_conjcar0}
	Let $X$ be a projective primitive symplectic variety over a field $F$ of characteristic $0$ such that $b_2(X) \geq 5$ and $X_{\overline{F}}$ has $\Q$-factorial terminal singularities. 
	\begin{enumerate}
	    \item The absolute movable cone conjecture as stated in \ref{conj_relKM}, item $(1)$, holds for $X$ and the $\Bir(X)$-action on $\clMov^+(X)$. More precisely, there exists a rational polyhedral cone $\Pi$ which is a fundamental domain under the action of $\Bir(X)$ on $\clMov^+(X)$.
	    \item The absolute nef cone conjecture as stated in \ref{conj_relKM}, item $(2)$, holds for $X$ and the $\Aut(X)$-action on $\Nef^+(X)$. More precisely, there exists a rational polyhedral cone $\Pi'$ which is a fundamental domain under the action of $\Aut(X)$ on $\Nef^+(X)$.
	\end{enumerate}
\end{thmx}

During the preparation of this work, Theorem \ref{thm_conjcar0} was stated in \cite[Theorem 3.5]{fu2025finiteness}, where a proof is briefly sketched. The approach developed here was obtained independently and follows similar lines. We hope that the detailed arguments provided in this paper, as well as the application presented below, will be a useful addition to the literature.

After the completion of this work, C. Gachet informed us that Theorem \ref{thm_conjcar0} also follows from her results \cite[Theorem 1.6 and Lemma 2.24]{gachet2025well}, which apply to varieties with well-clipped movable cones, a notion that she introduced in \cite[Definition 3.2]{gachet2025well}. In \cite[Example 3.10.(6)]{gachet2025well}, she observes that primitive symplectic varieties with terminal $\Q$-factorial singularities satisfy this property. We are grateful to her for pointing this out. While Theorem \ref{thm_conjcar0} can thus be deduced from her general framework, the present paper develops a more specialised study of primitive symplectic varieties over arbitrary fields of characteristic zero. In particular, we establish base change and descent properties, construct an integral quadratic form on the Picard group analogous to the Beauville–Bogomolov–Fujiki form over the complex numbers, and give an explicit description of the wall-and-chamber decomposition of the movable cone via reflections.

Theorem \ref{thm_conjcar0} is a generalization to the singular case of \cite[Theorem 1.0.5]{takamatsu2025finiteness} and its proof follows the same approach that goes back to \cite{sterk1985finiteness}, \cite{markman2011survey} and \cite{bright2020finiteness}. This theorem can be used to prove the next result that is a natural generalization of \cite[Theorem 1.1]{horing2024relative} to the setting of singular fibres.

\begin{thmx}\label{thm_relconjpsv}
Let $f : X \to S$ be a projective fibration between quasi-projective $\Q$-factorial normal varieties such that $X$ is klt. Suppose that
\begin{enumerate}
\item the very general fibre $X_s$ of $f$ is a projective primitive symplectic variety with $b_2(X_s) \geq 5$ and $\Q$-factorial terminal singularities,
\item the good minimal models exist for effective klt pairs on the very general fibre of $f$.
\end{enumerate}
Then the relative movable cone conjecture holds for $f : X \to S$ and, up to isomorphism over $S$, there are finitely many small $\Q$-factorial modifications $X'$ of $X$ over $S$, and the relative nef cone conjecture holds for each of them.
\end{thmx}

The proof of this theorem follows the same strategy as in \cite{horing2024relative}. The issue of the existence of both minimal models and good minimal models for primitive symplectic varieties is an active subject of research. The recent work by C. Onorati and \'A.D. R\'\i os Ortiz \cite[Theorem A]{onorati2025syz} provides a proof of the semiampleness of nef divisors on primitive symplectic varieties which are locally trivial deformations of $M_v(S,H)$, with $v = m \cdot w$ a non-primitive Mukai, under natural numerical hypothesis on $w$. Similarly, thanks to L. Buelli \cite[Theorem 5.2.1]{buelli2025locally}, the semiampleness of nef divisors is now known on primitive symplectic varieties which are locally trivial deformations of $K_v(A,H)$ under similar numerical assumptions. In particular, for fibrations whose very general fibre is a locally trivial deformation of these moduli spaces, our Theorem \ref{thm_relconjpsv} applies as soon as termination is known. By \cite[Theorem 1.3]{lehn2024footnotes}, any effective klt pair $(X,\Delta)$ with $X$ a projective primitive symplectic variety with terminal hyperquotient singularities and $\Delta$ an effective divisor has a minimal model. The termination is therefore known for Nikulin orbifolds \cite[Corollary 1.4]{lehn2024footnotes}, and G. Nanni proved in \cite{Nanni2025} the semiampleness of nef divisors on these primitive symplectic varieties, which provides an example to which our theorem \ref{thm_relconjpsv} applies.

\subsection*{Organization}
In Section \ref{section_psvdef}, we give the main definitions concerning primitive symplectic varieties over a field of characteristic 0. We prove some properties of base change with respect to field extensions for these varieties. In Section \ref{section_picardscheme}, we recall the definition and basic properties of the Picard scheme that we will need to apply the Lefschetz principle. This will enable us to equip the torsion-free part of the Picard group of a primitive symplectic variety defined over a field of characteristic zero with a lattice structure. In Section \ref{sect_absolute_prel}, we recall the definitions of the various cones with which we will be working and prove some preliminary results for the absolute conjectures. In Section \ref{section_movcone} and Section \ref{section_nefcone}, we prove Theorem \ref{thm_conjcar0} item (1) and Theorem \ref{thm_conjcar0} item (2) respectively. In Section \ref{sect_relative_prelim}, we recall some definitions from birational geometry and the Minimal Model Program that will be useful to us. Finally, in Section \ref{sect_relconeconj}, we prove Theorem \ref{thm_relconjpsv}.

\subsection*{Acknowledgements}
This paper is part of my doctoral work. I would like to thank my PhD advisor, Gianluca Pacienza, for his guidance and many helpful suggestions. I also thank Matei Toma and Benoît Cadorel for their help and for the many discussions that were of great benefit to me. Finally, I am grateful to Professors Fabio Bernasconi, Martin Bright, Christian Lehn, Adam Logan and Teppei Takamatsu for their kindness and for always taking the time to answer my questions by email.

\subsection*{Conventions}
All fields are assumed to be of characteristic zero, unless otherwise specified. For a field $F$, we denote by $\overline{F}$ an algebraic closure. A \emph{variety} is a geometrically integral, separated scheme of finite type over a field $F$. For any non-negative integer $i$ and for any prime number $\ell$, we denote by $b_i(X) := \dim_{\Q_\ell} H^i_{\etale}(X_{\overline{F}},\Q_\ell)$ the $i$-th Betti number of $X$ (this number is independant of $\ell$). A non-degenerate algebraic differential $2$-form on a variety is called \emph{symplectic} if it is closed with respect to the algebraic de Rham complex. A resolution of singularities of a variety $X$ is a proper birational morphism $Y \to X$, where $Y$ a non-singular variety. If $X$ and $S$ are varieties over a field $F$, a morphism $f : X \to S$ is called a \emph{fibration} if it is proper, surjective, and has connected fibres. When $F=\C$, by a very general fiber of $f$ we mean a fiber over a very general closed point.

\section{Definitions and first properties}
Throughout this section, we fix a field $F$ of characteristic $0$. If $F \subset L$ is a field extension and $X$ is a scheme over $F$, we will denote by $X_L$ the $L$-scheme $X \otimes_{F} L := X \times_{\Spec F} \Spec L$.

\subsection{Primitive symplectic varieties, base change and descent}\label{section_psvdef}
In this first subsection, we give the definition of a primitive symplectic variety. We will look at the first properties and establish some technical lemmas showing that this type of variety is stable by base change induced by field extensions (see Lemma \ref{lem_extofPSV} below for the exact statement), but also that it is possible to define them on a finitely generated subfield (cf. Corollary \ref{cor_descentofPSVtoQ}).

\begin{defn}\label{def_PSV}
    A \emph{projective primitive symplectic variety} over $F$ is a normal projective variety $X$ over $F$ of dimension $2n = \dim(X)$ such that
    \begin{itemize}
        \item there exists a non-zero closed $2$-form $\sigma \in \Gamma(X^{\reg},\Omega^2_{X/F})$ which is non-degenerate at any points of $X_{\reg}$, and a resolution of singularities $\pi : Y \to X$ such that $\pi^*\sigma$ extends to a global $2$-form on $Y$,
        \item $H^1(X,\Ox_{X}) = (0)$ and $H^0(X,\Omega^{[2]}_{X/F}) = F \cdot \sigma$, with $\Omega^{[2]}_{X/F} = (\Omega^2_{X/F})^{\vee\vee} \simeq \Omega^2_{X^{\reg}/F}$.
    \end{itemize}
\end{defn}
\begin{rmk}
    If there exists a resolution of singularities satisfying the pullback extension property of the symplectic $2$-form as above, then this remains true for any resolution. This is because two resolutions can always be dominated by a third.
\end{rmk}

We will need to manipulate various base changes induced by field extensions in order to bring us back to the well-known case of projective complex primitive symplectic varieties. A natural question then arises: is the base change of a projective primitive symplectic variety still a projective primitive symplectic variety? We can answer this question in the affirmative, as shown in the statement below.

\begin{lem}\label{lem_extofPSV}
    Let $F \subset L$ be a field extension and let $X$ be a projective primitive symplectic variety over $F$ of dimension $2n$. Then the base change $X_L$ is a projective primitive symplectic variety over $L$ of dimension $2n$.
\end{lem}
\begin{proof}
    First, let us note that $X_L$ is a projective variety over $L$ such that $\dim(X_L) = \dim(X)$ by base change. Furthermore, $X$ is automatically geometrically normal since $F$ is perfect \cite[\href{https://stacks.math.columbia.edu/tag/0C3M}{Tag 0C3M}]{stacks-project}, therefore $X_L$ is normal. By using the flat base change theorem \cite[\href{https://stacks.math.columbia.edu/tag/02KH}{Tag 02KH}]{stacks-project}, we obtain 
    \[H^1\left(X_L,\Ox_{X_L}\right) \simeq H^1\left(X,\Ox_X\right) \otimes_{F} L = (0).\] 
    
    Next, observe that the regular locus of $X$ commutes with the base change. Indeed, let $p : X_L \to X$ be the projection. By construction of the fibre product, we have $p^{-1}(X^{\reg}) \simeq (X^{\reg})_L$. We claim that the open subschemes $(X^{\reg})_L$ and $(X_L)^{\reg}$ of $X_L$ are equal. We simply need to show that they have the same underlying topological space. The set-theoretically inclusion $(X^{\reg})_L \subset (X_L)^{\reg}$ is provided, for example, by \cite[Proposition 5.1.(4)]{gortzwedhorn2}, and the reverse inclusion $(X_L)^{\reg} \subset (X^{\reg})_L$ is always true since $p$ is faithfully flat, see the discussion after \cite[Remark 18.50]{gortzwedhorn2}. From now on, we will denote the regular locus of $X_L$ by $X_L^{\reg}$.
    
    Finally, we must prove the existence of a symplectic 2-form on $X_L$ which satisfies the assumptions of Definition \ref{def_PSV}. Since the projection $p$ induces an isomorphism of complexes $p^*\Omega^\bullet_{X/F} \simeq \Omega^\bullet_{X_L/L}$ \cite[\href{https://stacks.math.columbia.edu/tag/07HX}{Tag 07HX}]{stacks-project}, we deduce that
    \[\Gamma(X_L^{\reg},\Omega^2_{X_L/L}) = L \cdot \sigma_L\]
    with $\sigma_L := p^*\sigma$ a symplectic 2-form. Let $\pi : Y \to X$ be a resolution of singularities. We claim that the base change $\pi_L : Y_L \to X_L$ is a resolution of singularities such that $\pi_L^* \sigma_L$ extends to a global $2$-form on $Y_L$. Indeed, $\pi$ is an isomorphism between dense open subsets $U \subset Y$ and $V \subset X$, so the same will apply to $\pi_L$ with $U_L \subset Y_L$ and $V_L \subset X_L$, which shows that $\pi_L$ is a birational morphism. Since smoothness and properness are stable under base change, we see that $\pi_L$ is a resolution of singularities. Denote by $q : Y_L \to Y$ the projection. Now observe that if $\omega \in H^0(Y,\Omega^2_{Y/F})$ is an extension of $\pi^*\sigma$, then $q^*\omega \in H^0(Y_L, \Omega^2_{Y_L/L})$ is an extension of $\pi_L^*\sigma_L$. This shows that $X_L$ is a primitive symplectic variety on $L$, thus completing the proof.
\end{proof}

\begin{lem}\label{lem_descentofpsvfg}
    Let $F'$ be a subfield of $F$ and let $X'$ be a variety over $F'$. If the base change $X := X'_{F}$ is a projective primitive symplectic variety over $F$, then $X'$ is a projective primitive symplectic variety over $F'$.
\end{lem}
\begin{proof}
    First of all, note that $X'$ is normal since $X$ is geometrically normal. Moreover, thanks to \cite[Corollaire 6.6.5]{grothendieck1961elements} (although this result is stated for finite field extensions, it remains true without this finiteness assumption), we know that $X'$ is projective. Now fix a resolution of singularities $\pi' : Y' \to X'$ with $Y'$ a smooth variety over $F'$. Let $Y := Y'_{F}$, and denote by $p_X : X \to X'$ and $p_Y : Y \to Y'$ the projections. By base change, we obtain a resolution of singularities $\pi : Y \to X$ such that $\pi = \pi'_F$. The flat base change theorem gives $h^1(X', \Ox_{X'}) = 0$ and $h^0(X', \Omega^{[2]}_{X'/F'}) = 1$. As in the previous proof, the base change induces an isomorphism of algebraic de Rham complexes, so we deduce the existence of a unique (up to multiplication by a nonzero scalar) symplectic 2-form $\sigma'$ on $X'^{\reg}$ such that $\sigma := p_X^* \sigma'$ is a symplectic 2-form on $X^{\reg}$. The only obstruction to the fact that $X'$ is a primitive symplectic primitive variety is that, a priori, $\pi'^*\sigma'$ does not extend to a global $2$-form over $Y'$. Set $U' := X'^{\reg}$ and $U := U'_{F} = X^{\reg}$. Consider the following commutative diagram.
    \begin{center}
        \begin{tikzcd}
{\Gamma(Y,\Omega^2_{Y/F})} \arrow[rr, "\operatorname{res}"]                            &  & {\Gamma(\pi^{-1}(U),\Omega^2_{Y/F})}                             &  & {\Gamma(U,\Omega^2_{X/F})} \arrow[ll, "\pi^*"']                             \\
                                                                                       &  &                                                                         &  &                                                                                    \\
{\Gamma(Y',\Omega^2_{Y'/F'})} \arrow[rr, "\operatorname{res}"] \arrow[uu, "p_Y ^*"] &  & {\Gamma(\pi'^{-1}(U'),\Omega^2_{Y'/F'})} \arrow[uu, "p_Y^*"] &  & {\Gamma(U',\Omega^2_{X'/F'})} \arrow[uu, "p_X^*"] \arrow[ll, "\pi'^*"']
\end{tikzcd}
    \end{center}
    By assumption, there exists $\omega \in \Gamma(Y,\Omega^2_{Y/F})$ such that $\omega \vert_{\pi^{-1}(U)} = \pi^*\sigma$. Let $C$ be the cokernel of the restriction map $\Gamma(Y',\Omega^2_{Y'/F'}) \to \Gamma(\pi'^{-1}(U'),\Omega^2_{Y'/F'})$ and consider $[\pi'^*\sigma'] \in C$. By the flat base change theorem, the top line of the diagram corresponds to the bottom line tensored by $F$, and under this identification, we have that $\pi^*\sigma = \pi'^*\sigma' \otimes 1$ and that $C \otimes_{F'} F$ is the cokernel of the restriction map $\Gamma(Y,\Omega^2_{Y/F}) \to \Gamma(\pi^{-1}(U),\Omega^2_{Y/F})$. We deduce that $[\pi'^*\sigma'] \otimes 1 = [\pi^*\sigma] = 0 \in C \otimes_{F'} F$, and therefore there exists $\omega' \in \Gamma(Y', \Omega^2_{Y'/F'})$ such that $\omega' \vert_{\pi'^{-1}(U')} = \pi'^*\sigma'$ (note that, by injectivity of the restriction $\Gamma(Y,\Omega^2_{Y/F}) \to \Gamma(\pi^{-1}(U),\Omega^2_{Y/F})$, we have $\omega = p_Y^*\omega'$). This shows that $X'$ is a primitive symplectic variety over $F'$, thus completing the proof.
\end{proof}
Although it is well known that a variety defined over a field is actually defined over a finitely generated subfield, we recall this result in the next proposition.

\begin{prop}[Proposition 5.6 - \cite{borovoi2020equivariant}]\label{prop_definedfgsubf}
    Let $k$ be a field. For any $k$-variety $X$ and any subfield $k_0 \subset k$, the following holds.
    \begin{enumerate}
        \item There exists a finitely generated sub-extension $k_0 \subset k_1 \subset k$ and a $k_1$-variety $X_1$ such that $X = (X_1)_k$. 
        \item If $k$ is a subfield of $k''$, $f : X'' \to Y''$ is a morphism of $k''$-varieties and $X$, $Y$ are two $k$-varieties such that $X'' = X_{k''}$ and $Y'' = Y_{k''}$, then there exists a finitely generated sub-extension $k \subset k' \subset k''$ and a $k'$-morphism $f' : X' \to Y'$ where $X' = X_{k'}$ and $Y' = Y'_{k'}$ such that $(X'',Y'',f'') \simeq (X', Y', f') \otimes_{k'} k''$.
    \end{enumerate}
\end{prop}
Therefore, combining Lemma \ref{lem_descentofpsvfg} and Proposition \ref{prop_definedfgsubf}, we obtain the following immediate corollary.

\begin{cor}\label{cor_descentofPSVtoQ}
    Let $X$ be a projective primitive symplectic variety over $F$. Then there exists a finitely generated subfield (over $\Q$) $F' \subset F$ and a projective primitive symplectic variety $X'$ over $F'$ such that $X \simeq X'_F$. We will say that $X$ is defined over $F'$.
\end{cor}
Another application of these base change results allows us to see that the singularities and global vector fields of a primitive symplectic variety over $F$ behave as in the case $F=\C$.

\begin{lem}\label{lem_tangent}
    The singularities of a projective primitive symplectic variety $X$ over $F$ are canonical. Furthermore $H^0(X,\cT_{X}) = 0$, where $\cT_{X} = \cT_{X/F}$ is the tangent sheaf of $X$ over $F$.
\end{lem}
\begin{proof}
    For the first claim, it seems well known that symplectic singularities are always canonical, but we will quickly review the proof. Let $\dim(X) = 2n$, $\sigma \in \Gamma(X^{\reg}, \Omega^2_{X/F})$ be a nonzero symplectic 2-form on $X$ and $\pi : Y \to X$ a resolution of singularities. Denote by $K_X$ and $K_Y$ the canonical classes on $X$ and $Y$ respectively. Since $\sigma$ is symplectic, $\sigma^n \in \Gamma(X^{\reg}, \Omega^{2n}_{X/F})$ defines a global section of $\Ox_X(K_X)$ that vanishes nowhere. Consequently, we have $K_X \sim_{\text{lin}} 0$. By definition of a projective primitive symplectic variety, the pullback $\pi^*\sigma$ extends to a regular global 2-form $\omega$ on $Y$ and $\operatorname{div}(\omega) = K_Y \geq 0$. Now, up to linear equivalence, we obtain
    \[K_Y = \sum_E a(E,X) E,\]
    where the sum runs through the exceptional divisors $E$. Since $K_Y$ is effective, we see that the discrepancies $a(E,X)$ are non-negative and therefore X has canonical singularities. 
    
    For the second claim, let $F' \subset F$ be a finitely generated subfield such that $X \simeq X'_F$ is defined over $F'$ as in Corollary \ref{cor_descentofPSVtoQ}. After fixing an embedding $F' \subset \C$, we know that $X'_\C$ is a projective primitive symplectic variety over $\C$ thanks to Lemma \ref{lem_extofPSV}. By \cite[Lemma 4.6]{bakker2022global}, we have $H^0(X'_\C, \cT_{X'_\C}) = 0$ and the conclusion follows from the flat base change theorem.
\end{proof}

\subsection{The Picard scheme and the Picard lattice of a primitive symplectic variety}\label{section_picardscheme}
In this subsection, we fix a field $F$ of characteristic $0$ and a projective primitive symplectic variety $X$ over $F$. We quickly recall the definition and some properties of the Picard scheme of $X$. This scheme will be useful for understanding the interaction between the Picard group of $X$ and the group of invariant elements under the action of the absolute Galois group $\Gal(\overline{F}/F)$ on $\Pic(X_{\overline{F}})$. We will use these results freely in the following without necessarily referring to them explicitly. For a more in-depth reading on the topic, see for example \cite[Part. 5]{fantechi2005fundamental}.

\begin{defn}
    For any $F$-schemes $T$, we set $X_T := X \times_{\Spec F} T$. The relative Picard functor Pic$_{X/F}$ is the contravariant functor defined by
    \[\Pic_{X/F} : (\text{Sch}/F ) \to (\text{Ab}), ~~\Pic_{X/F}(T) := \Pic(X_T)/p_T^*\Pic(T),\]
    where $p_T : X_T \to T$ is the projection. We denote its associated sheaf in the big étale topology by $\Pic_{(X/F)(\etale)}$.
\end{defn}

\begin{thm}[Theorem 9.4.8 - \cite{fantechi2005fundamental}]
    The sheaf $\Pic_{(X/F)(\etale)}$ is representable by a separated scheme $\bPic_{X/F}$ locally of finite type over $F$. Moreover, since $X/F$ is projective, this scheme is a disjoint union of open subschemes, each an increasing union of open quasi-projective $F$-schemes.
\end{thm}
Recall that the Brauer-Grothendieck group of a scheme $Y$ is $\Brauer(Y) = H^2_{\etale}(Y,\G_{m,Y})$. In the special case $Y = \Spec F$, this group is the usual Brauer group of $F$ in terms of Galois cohomology:
\[\Brauer(\Spec F) = \Brauer(F) = H^2\left(\Gal(\overline{F}/F), \overline{F}^\times \right).\] 

\begin{prop}[Proposition 4.3.2 - \cite{colliot2021brauer}]\label{prop_pic_brauer}
    Denote by $G_F = \Gal(\overline{F}/F)$ the absolute Galois group of $F$. The following sequence of abelian groups is exact.
    \begin{center}
        \begin{tikzcd}
0 \arrow[r] & \Pic(X) \arrow[r] & \Pic(X_{\overline{F}})^{G_F} \arrow[r] & \Brauer(F).
\end{tikzcd}
    \end{center}
\end{prop}

\begin{prop}[Corollary 4.1.3 - \cite{colliot2021brauer}]
    The abelian group $\Pic(X_{\overline{F}})$ is finitely generated.
\end{prop}

Therefore, both $\Pic(X)$ and $\Pic(X_{\overline{F}})^{G_F}$ are finitely generated, and since the Brauer group $\Brauer(F)$ is a torsion group (\cite[Corollary 1.3.6]{colliot2021brauer}), $\Pic(X)$ is of finite index in $\Pic(X_{\overline{F}})^{G_F}$. As we will see in Proposition \ref{prop_pairingetale}, we can define on $\Pic(X)_{\tf} := \Pic(X) / \operatorname{( torsion )}$ a structure of lattice. We briefly note that there is a natural inclusion $\Pic(X) \hookrightarrow \bPic_{X/F}(F)$, which is in general not an equality; the failure of equality is measured by the group $\Brauer(F)$ \cite[Corollary 2.5.9]{colliot2021brauer}.

\vspace{0.1cm}

Lemma \ref{lem_extofPSV} and Corollary \ref{cor_descentofPSVtoQ} are particularly useful for defining a symmetric bilinear pairing on a second étale cohomology group of a primitive symplectic variety. Indeed, when working with the field $F = \C$ and $Y$ a primitive symplectic variety over $\C$, we can look at the associated analytic primitive symplectic variety $Y^{\operatorname{an}}$. For such varieties, we know that there exists an integral nondegenerate quadratic form on the second cohomology group similar to the Beauville-Bogomolov-Fujiki form on irreducible holomorphic symplectic manifolds (also known as hyperkähler manifolds). The existence of such a form is the result of the work of several mathematicians, see for example \cite[$\S$5.1]{bakker2022global} for more information. Before giving the definition, we need a preliminary result.

\begin{prop}[Theorem 8 - \cite{schwald2020fujiki}]
    Let $Y$ be a complex projective klt variety. There are isomorphisms $H^{2,0}(Y) \simeq H^0(Y, \Omega^{[2]}_Y)$ and $H^{0,2}(Y) \simeq H^2(Y,\Ox_Y)$. For every resolution $\pi : Z \to Y$, the pullback morphism $\pi^* : H^{2,0}(Y) \to H^{2,0}(Z)$ is bijective.
\end{prop}
In particular, every reflexive two-form on $Y$ defines a unique class in $H^2(Y,\C)$. Since the singularities of a complex projective primitive symplectic variety $(Y,\sigma)$ are canonical by Lemma \ref{lem_tangent}, we will also denote by $\sigma$ the resulting cohomology class in $H^2(Y,\C)$ for simplicity. Below, we provide the definition of the Beauville-Bogomolov-Fujiki form as defined by T. Kirschner in \cite[Definition 3.2.7]{kirschner2015period}, but we note that there is also another definition given by Y. Namikawa in \cite[Theorem 8, (2)]{namikawa2001extension}. Fortunately, these two definitions are equivalent, see for example \cite[Corollary 24]{schwald2020fujiki} for a proof.

\begin{defprop}[Lemma 20, Definitions 19 and 21 - \cite{schwald2020fujiki}]\label{prop_bbf_analytic}
    Let $Y$ be a complex projective primitive symplectic variety of dimension $2m$ and let $\sigma \in H^{2,0}(Y)$ be the class of a holomorphic symplectic $2$-form on $Y^{\reg}$ satisfying $\int_Y (\sigma \cdot \overline{\sigma})^m = 1$. Then the quadratic form $q_Y : H^2(Y,\C) \to \C$ defined by
    \[q_Y(\alpha) := \dfrac{m}{2} \int_Y \alpha^2 \cdot (\sigma \cdot \overline{\sigma})^{m-1} + (1-m) \left( \int_Y \alpha \cdot \sigma \cdot (\sigma \cdot \overline{\sigma} )^{m-1} \right) \left( \int_Y \alpha \cdot \overline{\sigma} \cdot (\sigma \cdot \overline{\sigma})^{m-1} \right),\]
    does not depend on the choice of $\sigma$ and is called the Beauville-Bogomolov-Fujiki form on $Y$. It is nondegenerate and, up to scaling by a real number, is defined over $\Z$. Furthermore, the associated bilinear form has signature $(3, b_2(Y)-3)$, satisfies $q_Y \vert_{\left(H^{2,0}(Y) \oplus H^{0,2}(Y)\right)_\R} > 0$ and the orthogonal complement to $H^{2,0}(Y) \oplus H^{0,2}(Y)$ with respect to $q_Y$ is equal to $H^{1,1}(Y)$.
\end{defprop}
From now on, we will define an analogue of the Beauville-Bogomolov-Fujiki form on a projective primitive symplectic variety $X$ over a field $F$ of characteristic $0$. 

\begin{prop}
    Let $F$ be a field of characteristic $0$ and $X$ a primitive symplectic variety over $F$. Then for all but finitely many prime number $\ell \geq 2$, there exists a canonical injective morphism $c_1^{\ell} : \Pic(X)_{\tf} \hookrightarrow H^2_{\etale}(X_{\overline{F}},\Z_\ell)$.
\end{prop}
\begin{proof}
    Taking the cohomology of the Kummer sequence (\cite[p.77]{colliot2021brauer}), we obtain for any prime number $\ell$ and for every natural number $n$ an injective morphism $c_1^{\ell,n} : \Pic(X_{\overline{F}}) \otimes_\Z \Z/\ell^n \Z \hookrightarrow H^2_{\etale}(X_{\overline{F}}, \Z/\ell^n \Z)$. Since the abelian groups $\Z/\ell^n \Z$ are of finite length, they form with the natural projection maps a Mittag-Leffler system and by using the fact that $\Pic(X_{\overline{F}})$ is a finitely generated abelian group, we obtain a natural isomorphism $\varprojlim \Pic(X_{\overline{F}}) \otimes_\Z \Z/\ell^n \Z \simeq \Pic(X_{\overline{F}}) \otimes_\Z \Z_\ell$. Finally, we have an injective map
    \[c_1^\ell : \Pic(X_{\overline{F}}) \otimes_\Z \Z_\ell \hookrightarrow H^2_{\etale}(X_{\overline{F}},\Z_\ell)\]
    for any prime integer $\ell$, where $c_1^\ell := \varprojlim c_1^{\ell, n}$. If $e$ is the exponent of the (finite) torsion subgroup of $\Pic(X_{\overline{F}})$, then for $\ell > e$, we have
    \[\Pic(X_{\overline{F}})_{\tf} \otimes_\Z \Z_\ell \simeq \Pic(X_{\overline{F}}) \otimes_\Z \Z_\ell.\]
    For such $\ell$, it suffices to consider the restriction of $c_1^\ell$ to $\Pic(X)_{\tf} \hookrightarrow \Pic(X_{\overline{F}})_{\tf} \hookrightarrow \Pic(X_{\overline{F}})_{\tf} \otimes_\Z \Z_\ell$ to obtain the injective morphisms in the statement, which completes the proof.
\end{proof}
\begin{rmk}
    In fact, it is possible to define an injective morphism as in the statement for any prime number $\ell$, but it will no longer necessarily be canonical. To be more precise, since $\Pic(X_{\overline{F}})$ is finitely generated, the canonical exact sequence
    \[0 \to \operatorname{( torsion )} \to \Pic(X_{\overline{F}}) \to \Pic(X_{\overline{F}})_{\tf} \to 0\]
    always splits. After choosing a section $s$, simply precompose $c_1^\ell : \Pic(X_{\overline{F}}) \otimes_\Z \Z_\ell \hookrightarrow H^2_{\etale}(X_{\overline{F}}, \Z_\ell)$ with $s\otimes 1$ to obtain an injection. We preferred to base our proof without the choice of such a section, even if this meant that our statement did not cover all the prime integers.
\end{rmk}

\begin{prop}\label{prop_pairingetale}
    Let $F$ be a field of characteristic $0$, $G_F := \Gal(\overline{F}/F)$ the absolute Galois group of $F$ and $X$ a primitive symplectic variety over $F$ of dimension $2m$. Then, for all but finitely many prime number $\ell \geq 2$, there exists a unique $G_F$-invariant pairing $q_{X,\ell} : H^2_{\etale}(X_{\overline{F}},\Z_\ell)^{\otimes 2} \to \Z_\ell$ satisfying the following properties.
    \begin{enumerate}
        \item For any $\alpha_1, \dotsb, \alpha_{2m} \in H^2_{\etale}(X_{\overline{F}},\Z_\ell)$, we have
        \[(\alpha_1 \cdot \alpha_2 \cdot \dotsc \cdot \alpha_{2m}) = \dfrac{1}{m!2^m}\sum_{\sigma \in \sym_{2m}} q_{X,\ell}(\alpha_{\sigma(1)}, \alpha_{\sigma(2)})\dotsb q_{X,\ell}(\alpha_{\sigma(2m-1)}, \alpha_{\sigma(2m)}),\]
        where $\sym_{2m}$ is the symmetric group of order $2m!$, and the left term is the cup product of $\alpha_1, \dotsc, \alpha_{2m}$.
        
        \item For any $\alpha \in \Pic(X)_{\tf}$, $q_{X,\ell}\left(c_1^\ell(\alpha),c_1^\ell(\alpha)\right)$ lies in $\Z$ and is positive as soon as $\alpha$ is ample.
    \end{enumerate}
    Furthermore, the restrictions of $q_{X,\ell}$ to $\Pic(X)_{\tf}^{\otimes 2}$ via $c_1^\ell$ do not depend on $\ell$. The resulting morphism will be denoted $q_X$.
\end{prop}
\begin{proof}
    By Corollary \ref{cor_descentofPSVtoQ}, choose a finitely generated subfield $F' \subset F$ over $\Q$ and $X'$ a primitive symplectic variety over $F'$ such that $X = X'_{F}$. By fixing an embedding $\overline{F'} \subset \C$, we know from Lemma \ref{lem_extofPSV} that $X'_\C$ is a primitive symplectic variety over $\C$. By \cite[VI - Corollary 4.3]{milne1980etale}, for any prime integer $\ell$, we have $H^2_{\etale}(X_{\overline{F}},\Z_\ell) \simeq H^2_{\etale}(X'_{\overline{F'}},\Z_\ell) \simeq H^2_{\etale}(X'_\C,\Z_\ell)$ and by the comparison theorem \cite[Theorem 2]{artin1966etale}, we obtain an isomorphism $H^2_{\etale}(X'_{\overline{F'}}, \Z_\ell) \simeq H^2(X'^{\operatorname{an}}_{\C}, \Z) \otimes_\Z \Z_\ell$. We define $q_{X,\ell}$ thanks to this isomorphism and Proposition \ref{prop_bbf_analytic}. The rest of the proof is then identical to that of \cite[Proposition 2.1.5]{YANG2023108930} where the key argument is that the Beauville-Bogomolov-Fujiki form on an irreducible holomorphic symplectic manifold satisfies Fujiki relation, a relation which we also have for complex primitive symplectic varieties, see \cite[Theorem 2]{schwald2020fujiki}. In particular, the definition of $q_{X,\ell}$ is independent of the choice of an embedding $\overline{F'} \subset \C$ and this form verifies the properties of the statement.
\end{proof}

\section{Absolute cone conjectures for terminal Q-factorial primitive symplectic varieties}
In this section, we prove the movable and nef cone conjectures for a projective primitive symplectic variety $X$ over a field $F$ of characteristic $0$ with $b_2(X) \geq 5$ and $\Q$-factorial terminal singularities. Since $X$ is integral by definition, we will slightly abuse notation by identifying the notion of classes of Cartier divisors (for the linear equivalence relation) with the notion of classes of line bundles (for the isomorphism relation).

\subsection{Preliminaries}\label{sect_absolute_prel}
First, we recall the definitions and properties of the various cones that will be discussed in the rest of this paper. 

\begin{defn}
    Let $X$ be a primitive symplectic variety over a field $F$ of characteristic $0$ and let $\Pic(X)_\R := \Pic(X)\otimes_\Z \R$ be the \emph{real Picard space} of $X$. Inside this vector space, we consider the following convex cones.
    \begin{itemize}
        \item The ample cone $\Amp(X)$ of $X$, generated by ample Cartier divisors on $X$.
        \vspace{0.1cm}
        
        \item The nef cone $\Nef(X)$ of $X$, which is the closure of the ample cone. We denote by $\Nef^+(X)$ the convex hull of $\Nef(X) \cap \Pic(X)_\Q$ in $\Pic(X)_\R$.
        \vspace{0.1cm}
        
        \item The positive cone $\Pos(X)$ of $X$, which is the connected component of the cone of positive vectors of $\Pic(X)_\R$ (with respect to $q_X$) containing $\Amp(X)$. We denote by $\clPos(X)^+$ the convex hull of $\clPos(X) \cap \Pic(X)_\Q$ in $\Pic(X)_\R$.
        \vspace{0.1cm}

        \item The movable cone $\Mov(X)$ of $X$, generated by the classes of all effective integral divisors having a base locus of codimension at least $2$. Its closure is the cone $\clMov(X)$, its interior is the cone $\Mov^\circ(X)$ and $\clMov^+(X)$ is the convex hull of $\clMov(X) \cap \Pic(X)_\Q$ in $\Pic(X)_\R$.
        \vspace{0.1cm}
        
        \item The big cone $\Big(X)$ of $X$, which is the cone spanned by all big divisors. Recall that a Cartier divisor $D$ on $X$ is big if $h^0(X,\Ox_X(kD)) > c \cdot k^{\dim(X)}$ for some $c > 0$ and $k \gg 1$.
    \end{itemize} 
\end{defn}
\begin{rmk}
	Alternatively, we may work with the Neron-Severi space instead of the Picard space. Indeed, if $X$ is a primitive symplectic variety over a field $F$ of characteristic $0$, then $\Pic(X)_\mathbb{K} \simeq N^1(X)_\mathbb{K}$ for $\mathbb{K} = \Q$ or $\R$ by \cite[Corollary 9.5.13]{fantechi2005fundamental}.
\end{rmk}
This lemma will be used implicitly in the following and is certainly known to experts, but we have chosen to include a proof.

\begin{lem}\label{lem_amplemovbasechange}
    Let $F \subset L$ be an extension field, and $X$ a proper scheme of finite type over $F$. Denote by $p : X_L \to X$ the projection map, and suppose that $X_L$ is irreducible (in particular, $X$ is itself irreducible). Consider the following list of properties of divisors:
    \begin{itemize}
        \item $(P_1):$ ample,
        \item $(P_2):$ movable,
        \item $(P_3):$ big.
    \end{itemize}
    If $D \in \Pic(X)$ and $i \in \{1,2,3\}$, then $D$ is $(P_i)$ on $X$ if and only if $D_L := p^*D$ is $(P_i)$ on $X_L$.
\end{lem}
\begin{proof}
    First of all, the flat base change theorem immediately implies that $D$ is big if and only if $D_L$ is big. Next, for the part of the statement concerning ampleness, one direction is clear: if $D$ is ample, then $D_L$ is ample. For the converse implication, assume that $D_L$ is ample. We show that $D$ satisfies the cohomological ampleness criterion \cite[\href{https://stacks.math.columbia.edu/tag/0B5U}{Tag 0B5U}]{stacks-project}. Let $\cF$ be a coherent $\Ox_X$-module. For every integers $p,n\geq 0$, by the flat base change theorem, we have an isomorphism
    \begin{equation}\label{eq_cohomample}
        H^p\left(X_L,\cF_L \otimes D_L^{\otimes n}\right) \simeq H^p\left(X,\cF \otimes D^{\otimes n}\right) \otimes_F L.
    \end{equation}
    The pullback $\cF_L := p^*\cF$ being coherent, we know from the cohomological ampleness criterion that there exists $n_0 \geq 0$ such that for every $n \geq n_0$ and every $p \geq 1$, the groups in (\ref{eq_cohomample}) are trivial. It follows that $D$ is ample, since $L$ is faithfully flat over $F$.
    
    Finally, we show that $D$ is movable if and only if $D_L$ is movable. We denote by $\bsloc(D)$ and $\bsloc (D_L)$ the base loci of $\abs{D}$ and $\abs{D_L}$ respectively. Since we are only interested in the codimension, we equip these schemes with their reduced structure. We have the set-theoretic equality $p^{-1}\left( \bsloc (D) \right) = \bsloc (D_L)$ since $H^0(X_L, D_L) = H^0(X,D)\otimes_F L$. Consider the base change $q : \bsloc (D) \otimes_F L \to \bsloc (D)$. The map $Z \mapsto \overline{q(Z)}$ is a well-defined surjective map from the set of irreducible components of $\bsloc (D) \otimes_F L$ to the set of irreducible components of $\bsloc (D)$, and $\dim Z = \dim \overline{p(Z)}$, see for example \cite[Corollary 5.45, Exercise 5.12]{gortzwedhorn1}. Since $X$ and $X_L$ are irreducible, we have
    \[\dim Z + \codim_{X_L} Z = \dim X_L ~~~\text{and}~~~ \dim \overline{p(Z)} + \codim_{X} \overline{p(Z)} = \dim X\]
    by \cite[Proposition 5.30.(2)]{gortzwedhorn1}. Since $\dim X = \dim X_L$, it follows that $\codim_{X_L} Z = \codim_{X} \overline{p(Z)}$ for every irreducible components $Z$ of $\bsloc (D) \otimes_F L$. As the underlying topological space of $\bsloc (D) \otimes_F L$ is $p^{-1} \left(\bsloc (D)\right) = \bsloc (D_L)$, we deduce that
    \[\codim_X \bsloc (D) = \codim_{X_L} \bsloc (D_L).\]
    This equality shows that $D$ is movable if and only if $D_L$ is movable, and this completes the proof.
\end{proof}
The following lemma is an adaptation of \cite[Lemma 2.1.4]{takamatsu2025finiteness} to the singular case and allows us to compare the different cones of $X$ listed above with those of $X_{\overline{F}}$.

\begin{lem}\label{lem_extensioncones}
    Let $X$ be a projective primitive symplectic variety over a field $F$ of characteristic $0$. Denote by $\overline{X}$ the base change $X_{\overline{F}}$ of $X$ to $\overline{F}$, and by $G_F$ the absolute Galois group of $F$. We have the following equalities:
    \begin{enumerate}
        \item $\Amp(X) = \Amp(\overline{X}) \cap \Pic(X)_\R$.
        \item $\Mov(X) = \Mov(\overline{X}) \cap \Pic(X)_\R$.
        \item $\Big(X) = \Big(\overline{X}) \cap \Pic(X)_\R$
        \item $\Nef(X) = \Nef(\overline{X}) \cap \Pic(X)_\R$.
        \item $\clMov(X) = \clMov(\overline{X})\cap \Pic(X)_\R$.
    \end{enumerate}
\end{lem}
\begin{proof}
We only prove the case of ampleness, the other cases are treated in a similar manner. Let $x = \sum a_i x_i$ be an element of $\Amp(\overline{X}) \cap \Pic(X)_\R$, where the $a_i$ are real numbers and the $x_i$ are integral classes of ample line bundles on $\overline{X}$. Take a finite Galois extension $F \subset F'$ such that the $x_i$ are contained in $\Pic(X_{F'})_{\tf}$ and let $\Gamma := \Gal(F'/F)$. For each index $i$, we consider 
\[x_i' = \dfrac{1}{\abs{\Gamma}} \cdot \sum_{\sigma \in \Gamma} \sigma(x_i),\]
where $\abs{\Gamma}$ is the order of $\Gamma$. Note that the $x_i'$, seen as elements of $\Pic(\overline{X})_{\tf}$, are $G_F$-invariant and are still classes of ample line bundles on $\overline{X}$. Therefore, using the fact that $\Pic(X)$ is of finite index in $\Pic(\overline{X})^{G_F}$ and Lemma \ref{lem_amplemovbasechange}, it follows that the $x_i'$ are contained in $\Amp(X)$. Since $x$ is $G_F$-invariant, we have $x = \sum a_i x_i'$, and thus $x$ lies in $\Amp(X)$. Finally, the last two equalities can be deduced from the following lemma.
\end{proof}

\begin{lem}[Lemma 3.8 - \cite{bright2020finiteness}]\label{lem_coneboundary}
    Let $V$ be a real vector space, let $C \subset V$ be a convex cone, and let $S \subset V$ be a subspace having non-empty intersection with the interior of $C$. Then we have
    \[\partial_S(C \cap S) = \partial_V(C) \cap S.\]
\end{lem}
\begin{rmk}\label{rmk_invariantcommute}
    Lemma \ref{lem_coneboundary} is stated in \cite{bright2020finiteness} for \emph{closed} convex cones, but it remains true more generally for convex cones $C$ satisfying the assumption $S \cap \operatorname{int}(C) \neq \emptyset$. Also, note that the proof of Lemma \ref{lem_extensioncones} shows more generally that 
    \[\left(\Pic(\overline{X})_{\R}\right)^{G_F} = \left(\Pic(\overline{X})^{G_F}\right)_{\R} = \Pic(X)_\R.\]
\end{rmk}

\begin{defn}
    Let $X$ be a primitive symplectic variety over an algebraically closed field $L$ of characteristic $0$. A prime Weil divisor $E$ on $X$ is called a \emph{prime exceptional divisor} if it is $\Q$-Cartier and satisfies $q_X(E,E) < 0$. For any such prime exceptional divisor $E$, we define the associated reflection
    \[r_E : \Pic(X)_\Q \to \Pic(X)_\Q, ~~\alpha \mapsto \alpha - 2 \frac{q_X(E,\alpha)}{q_X(E,E)}E.\]
\end{defn}
We will show that these reflections are in fact integral when the singularities of $X$ are terminal, i.e. they are isometries of $\Pic(X)_{\tf}$. We will need two technical lemmas. 

\begin{lem}[Lemma 2.15 - \cite{fu2025finiteness}]\label{lem_qfactterm_basechange}
    Let $L \subset M$ be an extension of algebraically closed fields of characteristic $0$ and let $X$ be a normal projective variety over $L$. Then $X_M$ is a normal projective $\Q$-factorial klt variety if and only if $X$ is itself $\Q$-factorial and klt (the same statement remains true if we replace klt by terminal).
\end{lem}

\begin{lem}\label{lem_piciso}
    Let $L \subset M$ be an extension of algebraically closed fields of characteristic $0$, $X$ a projective primitive symplectic variety over $L$, and $X_M$ its base change to $M$. The projection $p : X_M \to X$ induces an isomorphism
    \[p^* : \Pic(X_M) \xrightarrow{\simeq} \Pic(X), ~L \mapsto p^*L.\]
    The induced isomorphism $p^*_{\tf} : \Pic(X_M)_{\tf} \xrightarrow{\simeq} \Pic(X)_{\tf}$ is an isometry with respect to the Beauville-Bogomolov-Fujiki forms on $X$ and $X_M$.
\end{lem}
\begin{proof}
    By \cite[Corollary 9.5.13]{fantechi2005fundamental} and Lemma \ref{lem_tangent}, the Picard scheme $\bPic_{X/L}$ is a smooth group scheme over $L$ of dimension $h^0(X,\cT_X) = 0$, and is therefore étale over $L$. Since $L$ is algebraically closed, we deduce that $\bPic_{X/L}$ is a disjoint union of copies of $\Spec L$, i.e. is a constant group scheme. In particular, the inclusion $\iota : L \hookrightarrow M$ induces an isomorphism
    \begin{equation}\label{eq_lempiciso}
        \bPic_{X/L}(\iota) : \bPic_{X/L}(L) \xrightarrow{\simeq} \bPic_{X/L}(M).
    \end{equation}
    Similarly, using the fact that $\bPic_{X_M/M} = \bPic_{X/L} \otimes_L M$, it is clear that $\bPic_{X_M/M}$ is also a constant group scheme such that $\bPic_{X_M/M}(M) = \bPic_{X/L}(M)$. By \cite[Corollary 2.5.8]{colliot2021brauer}, we have $\bPic_{X/L}(L) \simeq \Pic(X)$ and $\bPic_{X_M/M}(M) \simeq \Pic(X_M)$. We then observe that the isomorphism (\ref{eq_lempiciso}) corresponds to the pullback morphism $p^* : \Pic(X) \to \Pic(X_M)$. The fact that $p^*_{\tf}$ is an isometry is a consequence both of the description of this isomorphism as the pullback of line bundles, and of the construction of the Beauville-Bogomolov-Fujiki form through natural and functorial isomorphisms, and this concludes the proof.
\end{proof}

\begin{prop}\label{prop_isometry}
    Let $X$ be a primitive symplectic variety over an algebraically closed field $L$ of characteristic $0$ with terminal singularities. If $E$ is a prime exceptional divisor on $X$, then $r_E$ induces by restriction an isometry of $\Pic(X)_{\tf}$.
\end{prop}
\begin{proof}
    First, the case $L = \C$ is dealt with in \cite[Theorem 3.10]{lehn2024morrison}. Next, assume that $L \subset \C$ and denote by $p : X_{\C} \to X$ the projection. Note that $X_\C$ has terminal singularities thanks to Lemma \ref{lem_qfactterm_basechange}. Furthemore, if $E$ is a prime exceptional divisor, then its base change $E_\C \subset X_\C$ is still a prime exceptional divisor. Let $p^* : \Pic(X)_{\Q} \xrightarrow{\simeq} \Pic(X_\C)$ be the induced isometry of Lemma \ref{lem_piciso}. Since $r_E$ is equal to $(p^*)^{-1} \circ r_{E_\C} \circ p^*$ and since $r_{E_\C}$ induces by restriction an isometry of $\Pic(X_\C)_{\tf}$, it follows that $r_E$ induces an isometry of $\Pic(X)_{\tf}$.
    
    Finally, for the general case, take a finitely generated subfield $L' \subset L$ over $\Q$ and two primitive symplectic varieties $X'$ and $E'$ over $\overline{L'} \subset L$ such that $X = X'_L$ and $E = E'_L$ thanks to Lemma \ref{lem_extofPSV} and Corollary \ref{cor_descentofPSVtoQ}. By fixing an embedding $\overline{L'} \subset \C$, it follows that $r_{E'} \in O(\Pic(X')_{\tf})$ by the previous case. Since $r_E$ is conjugated as above to $r_{E'}$ by the isometry $\pi^* : \Pic(X')_{\Q} \xrightarrow{\simeq} \Pic(X)_{\Q}$ where $\pi : X \to X'$ is the projection, we conclude that $r_E \in O(\Pic(X)_{\tf})$, and this completes the proof.
\end{proof}
The following result will be useful.

\begin{lem}\label{lem_coeffentiers}
    Let $L$ be an algebraically closed field of characteristic $0$ and $X$ a projective primitive symplectic variety over $L$ with terminal singularities. For any prime exceptional divisor $E$ on $X$ and for any $\alpha \in \Pic(X)$, we have
    \[2\dfrac{q_X(E,\alpha)}{q_X(E,E)} \in \Z.\]
\end{lem}
\begin{proof}
    Take a finitely generated subfield $L' \subset L$ over $\Q$ such that both $X$ and $E$ are defined over $\overline{L'}$, i.e. $X = X'_L$ and $E = E'_L$ for $X'$ a projective primitive symplectic variety with terminal singularities (Lemma \ref{lem_qfactterm_basechange}) over $\overline{L'}$ and $E' \subset X'$ a prime exceptional divisor on $X'$. By fixing an embedding $\overline{L'} \subset \C$ and considering the projective primitive symplectic variety $X'_\C$ with the prime exceptional divisor $E'_\C$ on it, we see that it is sufficient to deal with the case $L = \C$ since $\Pic(X)_{\tf}$ and $\Pic(X'_\C)_{\tf}$ are isometric by Lemma \ref{lem_piciso}. We therefore assume that $L = \C$. We will prove something more general by replacing $\Pic(X)$ with $H^2(X,\Z)$ in the statement of this lemma. Following the notations of \cite[Remark 2.9]{lehn2024morrison}, for $\alpha \in H^2(X,\Q)$, we write $\alpha^\vee \in H_2(X,\Q)$ for the unique class satisfying $\alpha^\vee \cdot \beta = q_X(\alpha, \beta)$ for any $\beta \in H^2(X,\Q)$. In this way, we obtain an isomorphism $H^2(X,\Q) \to H_2(X,\Q), \alpha \mapsto \alpha^\vee$. Denote by $[E] \in H^2(X,\Q)$ the class of $E$. By \cite[Theorem 1.2]{lehn2021deformations}, $E$ is uniruled and the dual $[E]^\vee$ is proportional to the class of a general curve $R$ in a distinguished ruling. It is also noted in the proof of this theorem that $[E] \cdot [R] = -2$, so if $[E]^\vee = c [R]$ with $c \in \Q$, we deduce that
    \[c = \dfrac{-q_X(E,E)}{2}.\]
    It immediately follows that for any $\alpha \in H^2(X,\Z)$, we have
    \[2\dfrac{q_X(E,\alpha)}{q_X(E,E)} = 2c \dfrac{[R] \cdot \alpha}{q_X(E,E)} = -[R] \cdot \alpha \in \Z\]
    since $R$ defines an integral cohomology class, which concludes the proof.
\end{proof}

\begin{defn}\label{def_groupw}
    Let $L$ be an algebraically closed field of characteristic $0$ and $X$ a projective primitive symplectic variety over $L$. The subgroup of $O(\Pic(X)_{\tf})$ generated by the reflections with respect to prime exceptional divisors will be denoted by $W_X$. 
\end{defn}
\begin{rmk}
    For $L = \C$, the group $W_X$ is not defined in this way in \cite[Lemma 5.4]{lehn2024morrison}. Indeed, the authors define it as the subgroup of $\operatorname{Mon}^{2,\operatorname{lt}}_{\operatorname{Hdg}}(X) \subset O(H^2(X,\Z),q_X)$ generated by reflections by classes of stably exceptional line bundles, see \cite[Definition 2.16]{lehn2024morrison} for the notations. However, thanks to \cite[Theorem 5.12.(3), Theorem 5.12.(5), Lemma 6.3.(2)]{lehn2024morrison}, the image in $O(\Pic(X)_{\tf})$ of their group is isomorphic to our group, so their definition coincides in the Picard lattice of $X$ with the definition of $W_X$ given above.
\end{rmk}

\begin{prop}\label{prop_algcl_fundomain}
    Let $L$ be an algebraically closed field of characteristic $0$ and let $X$ be a projective primitive symplectic variety over $L$ with $b_2(X) \geq 5$ and $\Q$-factorial terminal singularities. Then $\clMov(X) \cap \Pos(X)$ is a fundamental domain under the action of $W_X$ on $\Pos(X)$, cut out by closed half-spaces associated to prime exceptional divisors.
\end{prop}
\begin{proof}
    The case $L = \C$ is dealt with in \cite[Lemma 6.2]{lehn2024morrison}. For the general case, take a finitely generated subfield $L' \subset L$ over $\Q$ and a $X'$ a primitive symplectic variety over $\overline{L'}$ such that $X = X'_L$. After fixing an embedding $\overline{L'} \subset \C$, Lemma \ref{lem_piciso} gives isometries 
    \[\Pic(X)_\R \simeq \Pic\left(X'\right)_\R \simeq \Pic\left(X'_\C\right)_\R\]
    which are induced by the pullbacks along the projections of the base changes. Observe that both $X'$ and $X'_\C$ have $\Q$-factorial and terminal singularities thanks to Lemma \ref{lem_qfactterm_basechange}, and that $b_2(X') = b_2(X'_\C) \geq 5$ by \cite[VI - Corollary 4.3]{milne1980etale}. The movable cones (respectively the positive cones) of $X$, $X'$ and $X'_\C$ correspond bijectively through the above isometries by Lemma \ref{lem_amplemovbasechange}, therefore the result follows from the case $L = \C$.
\end{proof}
For a projective primitive symplectic variety $X$ with terminal singularities over a field $F$ of characteristic $0$, denote by $G_F$ the absolute Galois group of $F$, and by $\overline{X}$ the base change $X_{\overline{F}}$. Let $W_{\overline{X}} \subset O\left(\Pic(\overline{X})_{\tf} \right)$ be the subgroup generated by reflections with respect to prime exceptional divisors on $\overline{X}$ as in Definition \ref{def_groupw}. $G_F$ acts naturally by automorphisms on $\Pic(\overline{X})$, and therefore acts naturally by conjugation on $O(\Pic(\overline{X})_{\tf})$:
\[\forall g \in G_F, \forall f \in O( \Pic(\overline{X})_{\tf}), \forall x \in \Pic(\overline{X})_{\tf} : (\leftindex^g f)(x) = g(f(g^{-1}x)).\]
Since $\leftindex^g r_E = r_{g E}$ for every prime exceptional divisors $E$ on $\overline{X}$, the action by conjugation defines an action of $G_F$ on $W_{\overline{X}}$.

\begin{defn}
    With the above notations, we denote the fixed part of $W_{\overline{F}}$ under this action by
    \[R_X := W_{\overline{X}}^{G_F} = \enstq{r \in  W_{\overline{X}}}{\forall g \in G_F: \leftindex^g r = r}.\]
\end{defn}
As explained in \cite{bright2020finiteness}, although it is straightforward to see that $R_X$ induces an action on $\Pic\left(\overline{X}\right)_{\tf}^{G_F}$, it is not entirely clear that this action preserves $\Pic(X)_{\tf}$. To prove this, we would like to adapt the arguments used to prove \cite[Proposition 3.6]{bright2020finiteness}. We will need the notion of a Coxeter system which we recall below.

\begin{defn}
    Let $W$ be a group with identity element $1$, and let $T \subset W$ be a generating set such that $T = T^{-1}$ and $1 \notin T$. For any $w \in W$, the length $l(w)$ of $w$ is the smallest integer $q \geq 0$ such that $w$ can be written as a product of $q$ elements of $T$. Assume that every element of $T$ has order $2$. For $t_i,t_j \in T$, let $n_{i,j} = n_{j,i}$ denote the order of $t_it_j$ if it is finite, and set $n_{i,j} = 0$ otherwise. We say that $(W,T)$ is a Coxeter system if $T$ together with the relations $(t_it_j)^{n_{i,j}} = 1$ for $n_{i,j} \neq 0$ gives a presentation of $W$.
\end{defn}

\begin{defn}
    The Coxeter–Dynkin diagram of a Coxeter system $(W,T)$ is the graph whose vertex set is $T$, with an edge between two distinct vertices $t_i$ and $t_j$ if and only if $t_i$ and $t_j$ do not commute. Each edge joining $t_i$ and $t_j$ is labelled by $n_{i,j} - 2$ if $n_{i,j} > 0$, and by $0$ otherwise.
\end{defn}

\newpage

\begin{prop}\label{prop_propertiesCoxeter}
    Let $(W,T)$ be a Coxeter system.
    \begin{enumerate}
        \item If $W$ is a finite group, there exists a unique $w_0 \in W$ such that for any $w \in W \backslash \{w_0\}$, we have $l(w) < l(w_0)$. We say that $w_0$ is the longest element of $(W,T)$. \cite[Proposition 2.3.1]{bjorner2005combinatorics}

        \vspace{0.1cm}

        \item\label{prop_propertiesCoxeter_2} If $I$ is a non-empty subset of $T$, let $W_I$ denote the subgroup of $W$ generated by the elements of $I$. Then $(W_I,I)$ is a Coxeter system. \cite[Proposition 2.4.1.(i)]{bjorner2005combinatorics}
    \end{enumerate}
\end{prop}

The strategy is then to adapt \cite[Proposition 4.1.3]{takamatsu2025finiteness}, relying on \cite[Proposition 3.6]{bright2020finiteness}, to the singular case. In the next proposition, we will make the slight abuse of notation of using the same letter to denote a set consisting of exceptional prime divisors and the set of associated reflections.

\begin{prop}\label{prop_tak413_sing}
    Let $F$ be a field of characteristic $0$, and let $X$ be a primitive symplectic variety over $F$ with $b_2(X) \geq 5$ such that $\overline{X} := X_{\overline{F}}$ has $\Q$-factorial terminal singularities. Denote by $G_F$ the absolute Galois group of $F$. Let $I$ be a Galois orbit of prime exceptional divisors on $\overline{X}$, and denote by $W_I$ the subgroup of $W_{\overline{X}}$ generated by the reflections with respect to the elements of $I$. Let $\Phi$ be the set of Galois orbits $I$ of prime exceptional divisors such that $W_I$ is finite, and let $\cP_{\overline{X}}$ be the set of prime exceptional divisors of $\overline{X}$. Then the following holds.

    \begin{enumerate}
        \item\label{prop_tak413_1} Both $(W_{\overline{X}}, \cP_{\overline{X}})$ and $(W_I, I)$ are Coxeter systems.
    
        \item\label{prop_tak413_2} If $W_I$ is finite, the orbit $I$ is necessarily of one of the following two forms.
        
        \begin{enumerate}
            \item\label{prop_tak413_2a} For any distinct $E_1, E_2$ in $I$, we have $q_{\overline{X}}(E_1,E_2) = 0$.
            
            \item\label{prop_tak413_2b} For any $E_1 \in I$, there exists a unique $E_1' \in I \backslash \{E_1\}$ such that $q_{\overline{X}}(E_1,E_1') \neq 0$. In this case, we have $q_{\overline{X}}(E_1,E_1)=-2q_{\overline{X}}(E_1,E_1')$.
        \end{enumerate}

        \item\label{prop_tak413_3} For each $I \in \Phi$, let $r_I$ be the longest element of the Coxeter system $(W_I, I)$ as in Proposition \ref{prop_propertiesCoxeter}, and let $E_I$ be the sum of the elements in $I$. Then $\left(R_X, \enstq{r_I}{I \in \Phi}\right)$ is a Coxeter system and $r_I$ acts on $\Pic\left(\overline{X}\right)^{G_F}_{\tf}$ as the reflection in the class of a sufficiently divisible integral multiple of $E_I$.

        \item\label{prop_tak413_4} The action of $R_X$ on $\Pic\left(\overline{X}\right)_{\tf}$ preserves $\Pic(X)_{\tf}$.

        \item\label{prop_tak413_5} A class $\alpha \in \Pos(X)$ lies in $\clMov(X)$ if and only if $q_{\overline{X}}(\alpha,E) \geq 0$ for all $I \in \Phi$ and $E \in I$.
    \end{enumerate}
\end{prop}
\begin{proof}~\

(1) This is a consequence of Proposition \ref{prop_propertiesCoxeter}.(\ref{prop_propertiesCoxeter_2}) and \cite[Section 5.4]{heckman2018coxeter}.

\vspace{0.1cm}

(2) Assume that $W_I$ is finite and let $E_1$ and $E_2$ be two distinct elements in $I$. Up to replacing $\overline{F}$ by a finitely generated subfield over which $\overline{X}$, $E_1$ and $E_2$ are defined, the Lefschetz principle allows us to assume that $\overline{F} =  \C$. In this case, the classes of $E_1$ and $E_2$ in $\Pic(\overline{X})_\R$ are different by \cite[Theorem 1.1]{kapustka2019boucksom}. To simplify the notation, let $\alpha = q_{\overline{X}}(E_1,E_2)$ and $\beta = q_{\overline{X}}(E_1,E_1) = q_{\overline{X}}(E_2,E_2)$. Reflections $r_{E_1}$ and $r_{E_2}$ stabilize the $2$-plane $W := \R E_1 \oplus \R E_2$. Restricting these reflections to $W$, we obtain the matrices
\[r_{E_1} = \begin{pmatrix}-1 & -2\alpha/\beta \\ 0 & 1\end{pmatrix}, ~~r_{E_2} = \begin{pmatrix}-1 & 0 \\ -2\alpha/\beta & -1\end{pmatrix} ~~\text{and}~~ r_{E_1}r_{E_2} = \begin{pmatrix}-1 + 4\alpha^2/\beta^2 & 2\alpha/\beta \\ -2\alpha/\beta & -1\end{pmatrix}.\]
These are matrices with integer coefficients by Lemma \ref{lem_coeffentiers} and, by assumption, they have finite order. By classical results in linear algebra, these orders belong to the set \{2,3,4,6\}. The order of $r_{E_1}r_{E_2}$ cannot be equal to either $4$ or $6$, as can be seen by computing the determinant and the trace from the eigenvalues, which are conjugate primitive roots of unity. A calculation shows that it is equal to $2$ if and only if $\alpha = 0$, and that it is equal to $3$ if and only if $\beta = -2\alpha$. By \cite[Exercise 1.4]{bjorner2005combinatorics}, the Coxeter–Dynkin diagram of $(W_I, I)$ is a finite union of finite trees, and we can therefore deduce that there is a vertex of degree at most $1$. But now the Galois group $G_F$ acts transitively on the diagram, so it follows that each vertex has degree at most $1$. On the one hand, saying that this degree is $0$ is equivalent, by definition of the Coxeter-Dynkin diagram, to the fact that the order of $r_{E_1}r_{E_2}$ is equal to $2$, i.e. $\alpha = 0$ and we are then in situation (\ref{prop_tak413_2a}) of the statement. On the other hand, saying that this degree is $1$ is equivalent to the fact that the order of $r_{E_1}r_{E_2}$ is equal to $3$, i.e. $\beta=-2\alpha$ and we are then in situation (\ref{prop_tak413_2b}) of the statement.

\vspace{0.1cm}

(3) $\left(R_X, \enstq{r_I}{I \in \Phi}\right)$ is a Coxeter system by \cite[Theorem 1]{geck2014coxeter}. As under the action of $r_I$, we prove it separately for the two types (\ref{prop_tak413_2a}) and (\ref{prop_tak413_2b}) of orbits. In the first case (\ref{prop_tak413_2a}), we have $I = \{E_1, \dotsc, E_r\}$ and $E_I = E_1 + \dotsb + E_r$. Replacing $E_I$ by a sufficiently divisible multiple, we may assume that $E_I$ is Cartier. The group $W_I$ is isomorphic to the Coxeter group $A_1^r$ and the longest element is therefore equal to $r_I = r_{E_1} \circ \dotsb \circ r_{E_r}$. For $\alpha \in \Pic(\overline{X})^{G_F}_{\tf}$, we obtain
\begin{equation}\label{eq_rI_first}
    r_I(\alpha) = \alpha - 2 \dfrac{q_{\overline{X}}(E_1,\alpha)}{q_{\overline{X}}(E_1,E_1)}E_I,
\end{equation}
that is, $r_I$ acts on $\Pic(\overline{X})^{G_F}_{\tf}$ as the reflection in the class of $E_I$. In the second case (\ref{prop_tak413_2b}), we have $I = \{E_1, E_1', \dotsb, E_r, E_r'\}$ with $q_{\overline{X}}(E_i,E_i') = 1$, and the other $q_{\overline{X}}$-intersections are zero. We have $E_I = E_1 + E_1' + \dotsb + E_r + E_r'$. Once again, we may assume that $E_I$ is Cartier. The group $W_I$ is isomorphic to the Coxeter group $A_2^r$ and the longest element is therefore equal to $r_I = r_{E_1+E_1'} \circ \dotsb \circ r_{E_r+E_r'}$. For $\alpha \in \Pic(\overline{X})^{G_F}$, we obtain
\begin{equation}\label{eq_rI_sec}
    r_I(\alpha) = \alpha - 2^r \dfrac{q_{\overline{X}}(E_1,\alpha)}{q_{\overline{X}}(E_1,E_1)}E_I.
\end{equation}
which coincides with the reflection of $\alpha$ with respect to the class of $E_I$.

\vspace{0.1cm}

(4) Now, each class $E_I$ is the class of a Galois-fixed divisor on $\overline{X}$, so lies in $\Pic(X)_{\tf}$. So in both cases of (\ref{prop_tak413_2}), the formulas (\ref{eq_rI_first}) and (\ref{eq_rI_sec}) combined with Lemma \ref{lem_coeffentiers} show that the reflections $R_I$ preserve $\Pic(X)_{\tf}$. Since these reflections are generators of $R_X$ by (\ref{prop_tak413_3}), we deduce that $R_X$ preserves $\Pic(X)_{\tf}$.

\vspace{0.1cm}

(5) Take $\alpha \in \Pos(X) \subset \Pos(\overline{X})$. By Lemma \ref{lem_extensioncones}, $\alpha$ lies in $\clMov(X)$ if and only if it lies in $\clMov(\overline{X})$. By Proposition \ref{prop_algcl_fundomain}, $\clMov(\overline{X}) \cap \Pos(\overline{X})$ is cut out by closed half-spaces associated to prime exceptional divisors of $\overline{X}$. If we can show that
\begin{equation}\label{eq_posalpha}
    \forall E \in \cP_{\overline{X}} : q_{\overline{X}}(\alpha, E) \geq 0,
\end{equation}
then we are done. Let $E$ be such a prime exceptional divisor. By our assumptions, we already know that the inequality is true if the orbit $I$ of $E$ under the action of the absolute Galois group is such that $W_I$ is finite, so let us assume that $W_I$ is infinite. By contradiction, suppose that $q_{\overline{X}}(E,\alpha) < 0$. Since $\Pos(X)$ is connected, there exists $\lambda \in \Pos(X)$ such that $q_{\overline{X}}(E,\lambda) = 0$. As $\lambda$ is invariant under the action of the absolute Galois group, $\lambda$ is $q_{\overline{X}}$-orthogonal to every element in $I$ and is therefore fixed by the action of $W_I$. Consequently, $\lambda$ is also $q_{\overline{X}}$-orthogonal to $w(E)$ for any $w \in W_I$. Note that the orbit $I$ is finite because $\Pic(\overline{X})$ is of finite type, so that the action of the Galois group $G_F$ factors through a finite quotient. By \cite[Theorem 1]{speyer2009powers} and noting that $wr_Ew^{-1} = r_{w(E)}$, we deduce that $W_I \cdot E = \enstq{w(E)}{w \in W_I}$ is an infinite set. Denote by $N$ the set of all the classes $\beta \in \Pic(\overline{X})_\R$ such that $q_{\overline{X}}(\beta, \beta) < 0$. From what we have just said, $\lambda$ is $q_{\overline{X}}$-orthogonal to an infinite number of elements of $N$. However, arguing as in \cite[Remark 8.2.3]{huybrechts2016lectures}, we obtain that the following union
\[\bigcup_{\alpha \in N} \left(\alpha^{\perp} \cap \Pos(\overline{X})\right) \subset \Pos(\overline{X})\]
is locally finite, which yields a contradiction.
\end{proof}
Now that we know that the action of $R_X$ on $\Pic\left( \overline{X} \right)_{\tf}$ induces by restriction an action on $\Pic(X)_{\tf}$, we can prove the analogue of \cite[Proposition 4.0.4]{takamatsu2025finiteness}.

\begin{prop}\label{prop_tak404_sing}
    Let $F$ be a field of characteristic $0$ and let $X$ be a projective primitive symplectic variety over $F$ with $b_2(X) \geq 5$ such that $\overline{X} := X_{\overline{F}}$ has $\Q$-factorial terminal singularities. Then $\clMov(X) \cap \Pos(X)$ is a fundamental domain under the action of $R_X$ on $\Pos(X)$.
\end{prop}
\begin{proof}
    We denote by $\overline{X}$ the base change $X_{\overline{F}}$, and by $G_F$ the absolute Galois group of $F$. First, we show that for every $\alpha \in \Pos(X)$, there exists $g \in R_X$ such that $g\alpha \in \clMov(X) \cap \Pos(X)$. 
    
    If the stabilizer of $\alpha$ under the action of $W_{\overline{X}}$ on $\Pic(\overline{X})_\R$ is trivial, there is a unique $g \in W_{\overline{X}}$ such that $g\alpha \in \clMov(\overline{X}) \cap \Pos(\overline{X})$ by Proposition \ref{prop_algcl_fundomain}. We claim that $g \in R_X$. Indeed, for any $\sigma \in G_F$, we have
    \[(\leftindex^\sigma g)(\alpha) = \sigma g(\sigma^{-1}(\alpha)) = \sigma(g\alpha) \in \clMov(\overline{X}) \cap \Pos(\overline{X})\]
    since $G_F$ preserves $\clMov(\overline{X}) \cap \Pos(\overline{X})$. By uniqueness of $g$, we obtain $\leftindex^\sigma g = g$, that is, $g$ lies in $R_X$. Therefore, $g \alpha \in (\Pic(\overline{X})_{\R})^{G_F} = \Pic(X)$ (see Remark \ref{rmk_invariantcommute}), and we conclude that $g\alpha \in \clMov(X) \cap \Pos(X)$ by Lemma \ref{lem_extensioncones}.
    
    If now $\alpha$ has a non-trivial stabilizer, we must have by Proposition \ref{prop_algcl_fundomain} that
    \begin{equation}\label{eq_prop404}
        \alpha \in \bigcup_{E \in \cP_{\overline{X}}} E^\perp \cap \Pos(\overline{X}),
    \end{equation}
    where $\cP_{\overline{X}}$ is the set of prime exceptional divisors of $\overline{X}$. Arguing as in \cite[Remark 8.2.3]{huybrechts2016lectures}, the above union is locally finite. Since $X$ admits an ample divisor, $\Pic(X)$ is not contained in any of the walls appearing in (\ref{eq_prop404}). This allows to construct a sequence $(\alpha_n)_{n \geq 0}$ of elements of $\Pos(X)$ tending to $\alpha$, and all lying in the interior of the same chamber of $\Pos(\overline{X})$ cut out by the walls above. From what we have said before, we deduce that there is a unique $g \in R_X$ such that $g\alpha_n \in \clMov(X) \cap \Pos(X)$ for every $n \geq 0$. By continuity, it follows that $g\alpha \in \clMov(X) \cap \Pos(X)$.
    
    Next, we show that for all non-trivial element $g \in R_X$, $\clMov(X) \cap \Pos(X)$ and $g \cdot \clMov(X) \cap \Pos(X)$ intersect only at their boundaries. Suppose that $\alpha \in \Pos(X)$ lies in $\clMov(X) \cap g\cdot \clMov(X)$ for some non-trivial element $g \in R_X$. By Proposition \ref{prop_algcl_fundomain}, $\alpha$ lies in the boundary of $\clMov(\overline{X})$, and Lemma \ref{lem_coneboundary} shows that $\alpha$ lies in the boundary of $\clMov(X)$, which concludes this proof.
\end{proof}

\subsection{The absolute movable cone conjecture}\label{section_movcone}
Before proving Theorem \ref{thm_conjcar0}.(1), we recall a few facts about the group $\Bir(X)$ of a projective variety $X$ over an algebraically closed field $L$ of characteristic zero. We denote by $\operatorname{(LNSch/L)}$ the full subcategory of $\operatorname{(Sch/L)}$ consisting of locally Noetherian schemes over $L$. In the following, we adopt the notations of \cite[Section 3]{jeremyblanc2017birational}.

\begin{defn}
    Let $X$ be a variety over an algebraically closed field $L$ of characteristic zero. We define the following functor.
    \[\Bir_{X/L} : \operatorname{(LNSch/L)} \to (\text{Set}), ~~\Bir_{X/L}(S) := 
    \left\{
    \begin{array}{c}
    \textup{S-birational transformations of $S \times X$ inducing an} \\
    \textup{isomorphism between two open subsets of $S \times X$,} \\
    \textup{whose projections on $S$ are surjective}
    \end{array}
    \right\}.
    \]
\end{defn}

\begin{lem}[Lemma 3.3 - \cite{jeremyblanc2017birational}]\label{lem_birgraph}
    Let $X$ be a variety over an algebraically closed field $L$ of characteristic zero. For every locally Noetherian scheme $S$ over $L$, we have a natural bijection
    \[\Bir_{X/L}(S) \xrightarrow{\simeq} \left\{
    \begin{array}{c}
    \textup{irreducible closed subsets $Y \subset S \times X \times X$ admitting} \\
    \textup{a dense open subset $W \subset Y$ such that the projection} \\
    \textup{$W \to S$ is surjective, and such that the two projections} \\
    \textup{$S \times X \times X \to S \times X$ restrict to open immersions $W \to S \times X$} \\
    \end{array}
    \right\}\]
    which sends $f \in \Bir_{X/L}(S)$ to $\Gamma_f := \overline{\enstq{(s, x, q(f(s,x))}{(s,x) \in \operatorname{dom}(f)}}$, where $q : S \times X \to X$ is the projection.
\end{lem}

\begin{defn}[Definition 2.1 - \cite{hanamura1987birational}]
    Let $X$ be a projective variety over an algebraically closed field $L$ of characteristic zero, and $S$ a locally Noetherian scheme over $L$. A \emph{flat family of birational transformations} of $X$ over $S$ is a closed subscheme $Z \subset S \times X \times X$ which is flat over $S$, and such that for every point $s \in S$, the fibre $Z_s$ is the graph of an element of $\Bir_X(k(s))$.
\end{defn}

\begin{prop}[Proposition 1.7 - \cite{hanamura1987birational}]
    Let $X$ be a projective variety over an algebraically closed field $L$ of characteristic $0$. The functor
    \[\Bir^{\text{flat}}_{X/L} : \operatorname{(LNSch/L)} \to (\text{Set}), ~~S \mapsto \enstq{f \in \Bir_{X/L}(S)}{\Gamma_f \text{ is a flat family of birational transformations}}\]
    is representable by a scheme $\bBir_L^{\textbf{flat}}(X)$, where $\Gamma_f$ is the graph of $f$ as defined in Lemma \ref{lem_birgraph}.
\end{prop}
In general, this scheme does not have a group scheme structure. However, we have the following result.

\begin{thm}[Theorem 2.1 - \cite{hanamura1988structure}]\label{thm_hanamura}
  Let $X$ be a non-uniruled projective variety over an algebraically closed field $L$ of characteristic zero. The following holds.

  \begin{enumerate}
      \item $\dim_L \bBir_L^{\textbf{flat}}(X) \leq \min\left\{\dim(X), q(X) \right\}$, where $q(X)$ denotes the irregularity of a non-singular model of $X$.

      \item There exists a projective variety $Y$ over $L$ (which may be taken to be non-singular), which is birationally equivalent to $X$, such that the associated scheme $\bBir^{\textbf{flat}}_L(Y)_{\text{red}}$ has a natural structure of a group scheme, locally of finite type over $L$.
  \end{enumerate}
\end{thm}
We can apply these general results to our case to obtain the following lemma.

\begin{lem}\label{lem_biriso}
    Let $X$ be a $\Q$-factorial primitive symplectic variety over an algebraically closed field $L$ of characteristic $0$. If $L \subset M$ is an algebraically closed extension, then the projection $X_M \to X$ induces a group isomorphism $\Bir(X) \simeq \Bir(X_M)$.
\end{lem}
\begin{proof}
    First, we show that $X$ is non-uniruled. Suppose by contradiction that there exist a variety $Y$ over $L$ and a dominant rational map $\phi : Y \times \P^1_L \dashrightarrow X$. Take a finitely generated subfield $L' \subset L$ such that $Y$, $\P^1_L$, $X$ and $\phi$ are defined over $L'$, and fix an embedding $\overline{L'} \subset \C$. The base change $\phi_\C : Y_\C \times \P^1_\C \dashrightarrow X_\C$ is a dominant rational map, which means that $X_\C$ is uniruled. Note that $X_\C$ is a projective primitive symplectic variety by Lemma \ref{lem_extofPSV} and Lemma \ref{lem_descentofpsvfg}. Let $\pi : \tilde{X} \to X_\C$ be a resolution of singularities. Since $X_\C$ is uniruled, $\tilde{X}$ is also uniruled, and therefore $h^0(\tilde{X}, \Ox_{\tilde{X}}(K_{\tilde{X}})) = 0$. This contradicts the fact that the pullback of the non-zero symplectic $2$-form on $X_\C$ extends to a regular $2$-form on $\tilde{X}$. Hence, $X$ is non-uniruled.
    
    Next, we apply Theorem \ref{thm_hanamura} to $X$. Given that $h^1(X,\Ox_X) = 0$ and that $X$ has rational singularities by Lemma \ref{lem_tangent} and \cite[Théorème 1]{elkik1981rationalite}, it follows that the dimension of $\bBir_L^{\textbf{flat}}(X)$ is zero. Let $Y$ be a smooth projective variety birationally equivalent to $X$ as in Theorem \ref{thm_hanamura}. By \cite[Exposé VI.B, Corollaire 1.6.1]{SGA3I}, $\bBir_L^{\textbf{flat}}(Y)_{\text{red}}$ is smooth over $L$. Since $\dim \bBir_L^{\textbf{flat}}(X) = \dim \bBir_L^{\textbf{flat}}(Y)$ by \cite[Corollary 2.8]{hanamura1987birational}, the group scheme $\bBir_L^{\textbf{flat}}(Y)_{\text{red}}$ is of dimension $0$ and is therefore étale over $L$, equivalently it is a constant group scheme. It follows that the morphism $\Spec M \to \Spec L$ induces an isomorphism
    \[\bBir_L^{\textbf{flat}}(Y)_{\text{red}}(M) \simeq \bBir_L^{\textbf{flat}}(Y)_{\text{red}}(L),\]
    which simplifies to an isomorphism $\Bir(Y_M) \simeq \Bir(Y)$ by definition of $\Bir^{\text{flat}}_{Y/L}$. Since $Y$ is birationally equivalent to $X$, these two groups are respectively isomorphic to $\Bir(X_M)$ and $\Bir(X)$, and this completes the proof.
\end{proof}
Recall that if $X$ and $X'$ are projective primitive symplectic varieties over a field $L$ of characteristic zero with $\Q$-factorial terminal singularities, then a birational map $\phi : X \dashrightarrow X'$ is an isomorphism in codimension one by \cite[Corollary 3.54]{Kollár_Mori_1998}. Therefore, if $X' = X$, then $\phi$ induces an isometry $\phi^* \in O(\Pic(X)_{\tf})$. In order to prove the birational cone conjecture, we need to show that $\Bir(X_{\overline{F}})$ is a finitely generated group. The strategy consists in adapting \cite[Lemma 2 and Theorem 2]{boissiere2012note} to the singular case, relying on the results of \cite{lehn2024morrison} to generalize \cite[Lemma 2.1 and Proposition 2.2]{sterk1985finiteness}.

\begin{prop}[Corollary 3.5 - \cite{bakker2022global}]
    Let $X$ be a complex projective primitive symplectic variety. Then $H^2(X,\Z)_{\tf}$ carries a pure Hodge structure.
\end{prop}
For a complex projective primitive symplectic variety $X$, the group of locally trivial monodromy operators preserving the pure Hodge structure on $H^2(X,\Z)_{\tf}$ as in \cite[Definition 2.16]{lehn2024morrison} will be denoted by $\monlthdg(X)$ .

\begin{thm}[Theorem 4.2 - \cite{lehn2024morrison}]\label{thm_monodromop}
    Let $\phi : X \dashrightarrow X'$ be a birational map between complex projective primitive symplectic varieties with $\Q$-factorial terminal singularities. Then $\phi^* \in \monlthdg(X)$.
\end{thm}
For a complex projective primitive symplectic varieties with $\Q$-factorial terminal singularities, fix a generator $\sigma \in H^{2,0}(X)$ and consider the group $\Bir^0(X) := \enstq{f \in \Bir(X)}{f^* \sigma = \sigma}$ of birational transformations preserving $\sigma$. For any birational automorphism $f \in \Bir(X)$, there exists a unique $\chi(f) \in \C^\times$ such that $f^* \sigma = \chi(f) \cdot \sigma$ by Theorem \ref{thm_monodromop}. This defines a character $\chi : \Bir(X) \to \C^\times$ such that $\Bir^0(X) = \ker(\chi)$; in particular, $\Bir^0(X)$ is a normal subgroup of $\Bir(X)$.

\begin{lem}\label{lem_cyclic}
    Let $X$ be a complex projective primitive symplectic variety with $\Q$-factorial terminal singularities and $b_2(X) \geq 5$. The quotient $\Bir(X)/\Bir^0(X)$ is a finite cyclic group. 
\end{lem}
\begin{proof}
We adapt the arguments of \cite[Lemma 2]{boissiere2012note}. Denote by $T(X) \subset H^2(X,\Z)$ the transcendental “lattice” of $X$, i.e. the smallest integral Hodge structure such that $T(X)_\C := T(X) \otimes_\Z \C$ contains $H^{2,0}(X)$. By \cite[Theorem 8]{schwald2016low} and arguing as in \cite[Lemma 3.1]{huybrechts2016lectures}, we have $T(X)_{\tf} = \Pic(X)_{\tf}^{\perp}$. Set $T(X)_\R := T(X) \otimes_\Z \R$ and $E := \left(H^{2,0}(X) \oplus H^{0,2}(X)\right) \cap H^2(X,\R)$. Denote by $\tau$ the rank of $T(X)$. By definition of the transcendental lattice and by Definition/Proposition \ref{prop_bbf_analytic}, $(T(X)_{\tf}, q_X)$ is non-degenerate and is of signature $(2, \tau-2)$. Furthermore, we have an orthogonal decomposition
\[T(X)_\R = E \oplus \left(T(X)_\R \cap H^{1,1}(X)\right)\]
such that $q_X$ is positive on $E$, and negative on $T(X)_\R \cap H^{1,1}(X)$. For $f \in \Bir(X)$, the induced isometry $f^* \in O(T(X)_\R)$ preserves this decomposition and restricts to an orthogonal transformation on each summand. Consequently, all eigenvalues of $f^*_\C \in O(T(X)_\C)$ have modulus $1$. Since they are algebraic integers, they are roots of unity so that the minimal polynomial of $f^*$ is a cyclotomic polynomial $\Phi_n$ for some integer $n$. As $\Phi_n$ divides the characteristic polynomial of $f^*$, the Euler number of $n$ is smaller or equal to $\tau$. This shows that $\chi(\Bir(X)) \subset \C^\times$ is a finite group.
\end{proof}

\begin{lem}\label{lem_birfg}
    Let $L$ be an algebraically closed field of characteristic $0$ and $X$ a projective primitive symplectic variety over $L$ with $\Q$-factorial terminal singularities and $b_2(X) \geq 5$. The birational group $\Bir(X)$ of $X$ is finitely generated.
\end{lem}
\begin{proof}
    Let $L_0 \subset L$ be a finitely generated subfield such that $X$ is defined over $L_0$ and denote by $X_0$ (respectively $X_{0,\C}$) the corresponding primitive symplectic variety over $L_0$ (respectively the base change of $X_0$ to $\C$). From Lemma \ref{lem_biriso}, we know that $\Bir(X) \simeq \Bir(X_{0,\C})$, so we can assume that $L = \C$. Thanks to Lemma \ref{lem_cyclic}, in order to show that $\Bir(X)$ is finitely generated, it is sufficient to show that $\Bir^0(X)$ is finitely generated. Consider the restriction morphism $\rho : \monlthdg \to O(\Pic(X)_{\tf})$ as well as
    \[\cC^{\#} := \Pos(X) \cap \Pic(X)_{\tf}, ~\text{ and }~ O^{\#}(\Pic(X)_{\tf}) := \enstq{g \in O(\Pic(X)_{\tf})}{g(\cC^{\#}) = \cC^{\#}}.\]
    By \cite[Lemma 6.3.(1)]{lehn2024morrison}, the image $\Gamma$ of $\rho$ is a finite index subgroup of $O^{\#}(\Pic(X)_{\tf})$ and is in particular finitely generated by \cite[Theorem 6.12]{borel1962arithmetic}. Using the notation from the proof of Lemma \ref{lem_cyclic}, define $\monlthdg(X)_T := \enstq{\phi \in \monlthdg(X)}{\phi \vert_{T(X)} = \id}$ and $\Gamma_T = \rho(\monlthdg(X)_T) \subset \Gamma$. The group $\Gamma_T$ is an arithmetic subgroup of $\Gamma$, so is finitely generated. Let $\cC \cP_X \subset \Pic(X)_\Q$ be the set of classes of prime exceptional divisors of $X$ and define
    \[\begin{array}{rcl}
    \Gamma_{\Bir} & := & \enstq{g \in \Gamma}{g(\cC \cP_X) = \cC \cP_X}, \\
    \Gamma_{T,\Bir} & := & \enstq{g \in \Gamma_T}{g(\cC \cP_X) = \cC \cP_X}.
    \end{array}\]
    By \cite[Lemma 6.3.(3), Theorem 6.4.(3)]{lehn2024morrison}, we have $\Gamma \simeq \rho(W_X) \rtimes \Gamma_{\Bir}$. Since $T(X)_{\tf}$ is the orthogonal complement of $\Pic(X)_{\tf}$ in $H^2(X,\Z)_{\tf}$, elements of $W_X$ act trivially on $T(X)_{\tf}$, therefore $\Gamma_T \simeq \rho(W_X) \rtimes \Gamma_{T,\Bir}$. Consequently, $\Gamma_{T,\Bir}$ is finitely generated as a quotient of the finitely generated group $\Gamma_T$. By \cite[Lemma 6.3.(2) and (3), Theorem 6.4.(3)]{lehn2024morrison}, we have $\ker(\rho) \subset \monltbir(X)$, where $\monltbir(X)$ is the subgroup of $\monlthdg$ of locally trivial monodromy operators induced by birational transformations of $X$, and moreover $\Gamma_{\Bir} \simeq \monltbir(X)/\ker(\rho)$. This gives us the following identifications
    \[
    \begin{array}{rcl}
        \Gamma_{T,\Bir} & \simeq & \enstq{g \in \monltbir(X)}{g\vert_{T(X)_{\tf}} = \id} / \enstq{g \in \monltbir(X)}{g\vert_{\Pic(X)_{\tf}} = \id}\\
                        & \simeq &  \enstq{f \in \Bir(X)}{f^*\vert_{T(X)_{\tf}} = \id} / \enstq{f \in \Bir(X)}{f^*\vert_{H^2(X,\Z)_{\tf}} = \id}.
    \end{array}
    \]
    Since $X$ is projective, we have $\Bir^0(X) \simeq \enstq{f \in \Bir(X)}{f^*\vert_{T(X)} = \id}$. Since $\Gamma_{T,\Bir}$ is finitely generated, if we show that
    \[\enstq{f \in \Bir(X)}{f^*\vert_{H^2(X,\Z)_{\tf}} = \id}\]
    is finitely generated, then we are done. In fact, this group is finite by \cite[Proposition 2.4]{oguiso2014autcy}. Indeed, let $H$ be a very ample line bundle on $X$ and let $f \in \Bir(X)$ be a birational transformation acting trivially on $H^2(X,\Z)$. In particular, we have $f^*c_1(H) = c_1(H)$, so $f \in \Aut(X)$. Let
    \[\widetilde{G} := \enstq{f \in \Aut(X)}{f^*c_1(H) = c_1(H)},\]
    so that $\enstq{f \in \Bir(X)}{f^*\vert_{H^2(X,\Z)} = \id} \subset \widetilde{G}$. If we show that $\widetilde{G}$ is finite, we are done. Let $X \hookrightarrow \P^N$ be the embedding defined by $\abs{H}$, where $N = \dim \abs{H}$. Denote by $[X]$ the Zariski-closed point of $\Hilb_{\P^N}$ corresponding to the embedding we have just written. It follows that $\widetilde{G} \subset \PGL(N)$ and that $\widetilde{G}$ is the stabilizer of $[X]$ under the natural action of $\PGL(N)$ on $\Hilb_{\P^N}$. Since this action is Zariski-continuous, it follows that $\widetilde{G}$ is Zariski-closed in $\PGL(N)$. But now $\PGL(N)$ is affine noetherian, so is $\widetilde{G}$, and since $\dim \Aut(X) = h^0(X, \cT_X) = 0$ by \cite[Lemma 4.6]{bakker2022global}, we have $\dim \widetilde{G} = 0$ as well. It follows that $\widetilde{G}$ is finite, and that completes the proof.
\end{proof}
For a projective primitive symplectic variety $X$ over $F$, denote by $\overline{X}$ the base change $X_{\overline{F}}$ of $X$ over $\overline{F}$. Note that the action of $\Bir(X)$ on $O(\Pic(\overline{X})_{\tf})$ by conjugation fixes $R_X$. This gives an action of $\Bir(X)$ on $R_X$, and a morphism from the associated semi-direct product $\Bir(X) \ltimes R_X$ to $O(\Pic(\overline{X})_{\tf})$. Since the two factors of the semi-direct product fix $\Pic(X)_{\tf}$, we obtain a morphism $\rho : \Bir(X) \ltimes R_X \to O(\Pic(X)_{\tf})$. The following proposition is the analogue, in the singular setting, of \cite[Proposition 4.1.1]{takamatsu2025finiteness} 

\begin{prop}\label{prop_imagerho}
    Let $F$ be a field of characteristic $0$ and $X$ a projective primitive symplectic variety with $b_2(X) \geq 5$ such that $\overline{X} := X_{\overline{F}}$ has $\Q$-factorial terminal singularities. The image of $\rho$ is a finite index subgroup and the kernel of $\rho$ is a finite subgroup of $\Aut(X) \subset \Bir(X)$.
\end{prop}
\begin{proof}
    First, assume that $F$ is algebraically closed. Fix a finitely generated subfield $F' \subset F$ such that $X$ is the base change of a projective variety $X'$ over $\overline{F'} \subset F$. After fixing an embedding $\overline{F'} \subset \C$, we denote by $X'_\C$ the base change of $X'$ to $\C$. Note that both $X'$ and $X'_\C$ are projective primitive symplectic varieties with $\Q$-factorial terminal singularities thanks to Lemma \ref{lem_extofPSV}, Lemma \ref{lem_descentofpsvfg} and Lemma \ref{lem_qfactterm_basechange}. Moreover, $b_2(X') = b_2(X'_{\C}) \geq 5$ thanks to \cite[VI - Corollary 4.3]{milne1980etale}. Now observe that
    \[\Pic(X)_\R \simeq \Pic(X')_\R \simeq \Pic(X'_\C)_\R\]
    by Lemma \ref{lem_piciso}, and that $R_X \simeq R_{X'} \simeq R_{X'_\C}$ since, in our case, the projections $X \to X'$ and $X'_{\C} \to X'$ induce bijections between the sets of prime exceptional divisors. Moreover, we also have
    \[\Bir(X) \simeq \Bir(X') \simeq \Bir(X'_\C)\]
    thanks to Lemma \ref{lem_biriso}. Therefore, we are reduced to the case $F = \C$. In this case, the image of $\rho$ is a finite index subgroup by \cite[Theorem 5.12 and Lemma 6.3]{lehn2024morrison}. Concerning the kernel of $\rho$, recall that a birational morphism that fixes an ample line bundle under pullback extends to an automorphism, so that $\ker(\rho) \subset \Aut(X)$. Moreover, the group scheme $G$ of automorphisms of a projective variety fixing a quasi-polarization is a linear algebraic group as shown by \cite[Proposition 2.26]{brion2018notes}, so is of finite type over $\C$. Since $h^0(X,\cT_X) = 0$ by \cite[Lemma 4.6]{bakker2022global}, $G$ is of dimension zero and is therefore finite, so that $\ker \rho \subset G$ is also finite.
    
    We now turn to the case of a non-algebraically closed field $F$. Denote by $G_F$ the absolute Galois group of $F$. Let $\overline{\rho} : \Bir(\overline{X}) \ltimes R_{\overline{X}} \to O(\Pic(\overline{X})_{\tf})$. By the previous case, the kernel of $\overline{\rho}$ is finite and its image is of finite index in $O(\Pic(\overline{X})_{\tf})$. Note that $\Bir(\overline{X})^{G_F} = \Bir(X)$. The action of $G_F$ on $\Bir(\overline{X})$ factors through a finite quotient since $\Bir(\overline{X})$ is finitely generated by Lemma \ref{lem_birfg}. Similarly, as $\Pic(\overline{X})_{\tf}$ is finitely generated, the actions of $G_F$ on $R_{\overline{X}} = W_{\overline{X}}$ and $O(\Pic(\overline{X})_{\tf})$ factors through a finite quotient. Therefore, we can apply \cite[Lemma 3.12]{bright2020finiteness} to $\overline{\rho}$ and it follows that the induced morphism
    \[\Bir(X) \ltimes R_X \to O(\Pic(\overline{X})_{\tf})^{G_F}\]
    has a finite kernel, and its image is of finite index. Denote by $O(\Pic(\overline{X})_{\tf}, \Pic(X)_{\tf})$ the subgroup of $O(\Pic(\overline{X})_{\tf})$ stabilizing $\Pic(X)_{\tf}$ as a set. By \cite[Proposition 2.2.(2)]{bright2020finiteness}, $O(\Pic(\overline{X})_{\tf})^{G_F}$ is of finite index in $O(\Pic(\overline{X})_{\tf}, \Pic(X)_{\tf})$. Finally, by \cite[Proposition 2.2.(1)]{bright2020finiteness}, the morphism 
    \[O(\Pic(\overline{X})_{\tf}, \Pic(X)_{\tf}) \to O(\Pic(X)_{\tf})\]
    has a finite kernel, and its image is of finite index. We then check directly that the composite
    \[\Bir(X) \ltimes R_X \to O(\Pic(\overline{X})_{\tf})^{G_F} \hookrightarrow O(\Pic(\overline{X})_{\tf}, \Pic(X)_{\tf}) \to O(\Pic(X)_{\tf})\]
    coincides with $\rho$, which therefore has a finite kernel and a finite index image as composed of such morphisms. This completes the proof.
\end{proof}
Thanks to the results of this section, we will now prove Theorem \ref{thm_conjcar0}.(1). The strategy consists of adapting the proofs of \cite[Theorem 4.1.4]{takamatsu2025finiteness} and \cite[Corollary 3.15]{bright2020finiteness}. 

\begin{proof}[Proof of Theorem \ref{thm_conjcar0}.(1)]
    Denote by $\Gamma$ the image of $\rho$. Since $\Gamma$ is of finite index in $O(\Pic(X)_{\tf})$ by Proposition \ref{prop_imagerho}, it is an arithmetic subgroup of $O(\Pic(X)_{\tf})$. Pick any ample class $y \in \Pos(X) \cap \Pic(X)_{\tf}$. Arguing as \cite[page 511]{sterk1985finiteness}, we obtain that the set
    \[\Pi := \enstq{x \in \Pos(X)^+}{\forall \gamma \in \Gamma : q_X(\gamma x, y) \geq q_X(x,y)}\]
    is a rational polyhedral fundamental domain under the action of $\Gamma$ on $\clPos^+(X)$. We show that $\Pi \subset \clMov(X)^+$. Take $x \in \Pi$. By Proposition \ref{prop_tak413_sing}.\ref{prop_tak413_5}, it is sufficient to show that $q_{\overline{X}}(x,E) \geq 0$ for any $E$ in a Galois orbit $I$ of prime exceptional divisors on $\overline{X}$ such that $W_I$ is finite, where $\overline{X}$ is the base change of $X$ to an algebraic closure $\overline{F}$ of $F$. Indeed, since $\Pi$ is rational polyhedral, we will automatically obtain the inclusion in $\clMov^+(X)$. For such an orbit $I$, after replacing the sum $E_I$ of the elements in $I$ by a sufficiently divisible multiple, we may assume that $E_I$ is Cartier. The longest element $r_I$ of the Coxeter system $(W_I,I)$ acts on $\Pic(X)_{\tf}$ as the reflection in the class of $E_I$. Since $r_I \in \Gamma$ and $x \in \Pi$, we obtain $q_{\overline{X}}(r_I(x),y) \geq q_X(x,y)$, which is equivalent to $q_{\overline{X}}(x,E_I) = \abs{I} \cdot q_{\overline{X}}(x,E) \geq 0$ for any $E \in I$.
    
    Denote by $\Gamma_{\Bir}$ the image of the morphism $\Bir(X) \to O(\Pic(X)_\R)$. We show that $\Pi$ is a fundamental domain under the action of $\Gamma_{\Bir}$ on $\clMov^+(X)$. By the above discussion, for any $x \in \clMov^+(X)$, we can find $f \in \Bir(X)$ and $r \in R_X$ such that $rf^*(x) \in \Pi \subset \clMov(X)$. Note that Proposition \ref{prop_tak404_sing} implies that $\clMov^+(X)$ is a fundamental domain under the action of $R_X$ on $\Pos(X)^+$. Observe that $f^*(x)$ already lies in $\clMov^+(X)$, therefore $f^*(x) = rf^*(x) \in \Pi$, and that concludes the proof.
\end{proof}

\subsection{The absolute nef cone conjecture}\label{section_nefcone}
In this section, we prove Theorem \ref{thm_conjcar0}.(2).

\begin{defn}\label{defn_bamp}
    Let $F$ be a field of characteristic $0$ and $X$ a projective primitive symplectic variety over $F$. If $X_{\overline{F}}$ has $\Q$-factorial and terminal singularities, we define \emph{the birational ample cone} of $X$ as the subset of $\Pic(X)_\R$ given by
    \[\BAmp(X) := \bigcup_{f : X \dashrightarrow Y} f^* \Amp(Y),\]
    where $f$ runs through all birational maps with target a projective primitive symplectic variety $Y$ over $F$ such that $Y_{\overline{F}}$ has $\Q$-factorial and terminal singularities.
\end{defn}
\begin{rmk}
    When $F = \C$, we recover the classical definition of the birational ample cone as written, for example, in \cite[Definition 4.5]{lehn2024morrison}. Also, note that if $f : X \dashrightarrow Y$ is a birational map as in the definition above, then $f$ is an isomorphism in codimension one, and hence the pullback of line bundles is well defined. Indeed, the singularities of $X$ and $Y$ are terminal if and only if the singularities of $X_{\overline{F}}$ and $Y_{\overline{F}}$ are terminal by \cite[Proposition 2.15]{kollar2013singularities}. Therefore, the claim follows from \cite[Corollary 3.54]{Kollár_Mori_1998}. 
\end{rmk}

\begin{lem}\label{lem_bampinter}
    Let $F$ be a a field of characteristic $0$ and $X$ a projective primitive symplectic variety over $F$ such that $\overline{X} := X_{\overline{F}}$ has $\Q$-factorial and terminal singularities. Then $\BAmp(X) = \BAmp(\overline{X}) \cap \Pic(X)_\R$.
\end{lem}
\begin{proof}
    For the inclusion $\BAmp(X) \subseteq \BAmp(\overline{X}) \cap \Pic(X)_\R$, let $f : X \dashrightarrow Y$ be a birational map as in Definition \ref{defn_bamp}, and take $\alpha \in \Amp(Y)$. Set $\overline{Y} := Y_{\overline{F}}$. Consider the Cartesian square associated with the base change of $f$ to $\overline{F}$, where the vertical arrows are the canonical projections.
    \begin{center}
        \begin{tikzcd}
\overline{X} \arrow[dd, "p_X"] \arrow[rr, "f_{\overline{F}}", dashed] &         & \overline{Y} \arrow[dd, "p_Y"] \\
                                                                          & \square &                                    \\
X \arrow[rr, "f", dashed]                                                 &         & Y                                 
\end{tikzcd}
    \end{center}
    We have $p_X^* f^* (\alpha) = f_{\overline{F}}^* p_Y^* (\alpha)$, and it follows from Lemma \ref{lem_amplemovbasechange} that 
    \[p_X^* f^*(\alpha) \in \BAmp(\overline{X}) \cap \Pic(X)_\R,\]
    which proves the desired inclusion. 
    
    For the reverse inclusion $\BAmp(\overline{X}) \cap \Pic(X)_\R \subseteq \BAmp(X)$, let $G_F$ be the absolute Galois group of $F$. Take $\overline{x} \in \BAmp(\overline{X}) \cap \Pic(X)_\R$ and let $\overline{f} : \overline{X} \dashrightarrow \overline{Y}$ be a birational map as in Definition \ref{defn_bamp} such that $\overline{x} = \overline{f}^*(\overline{y})$, with $\overline{y} \in \Amp(\overline{Y})$. We claim that if there exists a birational map $f: X \dashrightarrow Y$ over $F$ such that $\overline{f} = f_{\overline{F}}$, then we have proven the inclusion. Indeed, since both $\overline{x}$ and $\overline{f}$ are $G_F$-equivariant (being defined on $F$), $\overline{y}$ is also $G_F$-invariant and it follows that $\overline{y} \in \Amp(Y)$ by Lemma \ref{lem_extensioncones} and Remark \ref{rmk_invariantcommute}. 
    
    We must therefore show that $\overline{f}$ descends to a birational map over $F$. For definitions and terminology on Galois descent, see \cite[$\mathsection$ 6.2]{bosch2012neron}. Since $X$ is defined over $F$, we have a canonical descent data: for any $\sigma \in G_F$, there exists an isomorphism of schemes $\phi_\sigma : \overline{X} \to \overline{X}$ such that the diagram
    \begin{center}\label{eq_diagdescent}
        \begin{tikzcd}
\overline{X} \arrow[dd] \arrow[rr, "\phi_\sigma"] &  & \overline{X} \arrow[dd] \\
                                                      &  &                             \\
\Spec(\overline{F}) \arrow[rr, "\sigma"]            &  & \Spec(\overline{F})        
\end{tikzcd}
    \end{center}
    is commutative, and $p_X \circ \phi_\sigma = p_X$. For all $\sigma \in G_F$, define the birational map
    \[\psi_{\sigma} := \overline{f} \circ \phi_{\sigma} \circ \overline{f}^{-1} : \overline{Y} \dashrightarrow \overline{Y}.\]
    Observe that the $\psi_\sigma$ fit into a commutative diagram as above, and that they are isomorphisms in codimension one. If we can show that these birational applications extend to isomorphisms, then they will define a descent data for both $\overline{Y}$ and $\overline{f}$. Fix $\sigma \in G_F$. Since $\overline{x} \in \Pic(X)_\R$ and $p_X \circ \phi_\sigma = p_X$, we have $\phi_{\sigma}^*(\overline{x}) = \overline{x}$, so that
    \[\psi_\sigma^*(\overline{y}) = (\overline{f}^{-1})^*\phi_{\sigma}^*(\overline{x}) = (\overline{f}^{-1})^*(\overline{x}) = \overline{y}.\]
    Hence each $\psi_\sigma^*$ fixes an ample class and therefore extends to an isomorphism. This defines a descent datum for $\overline{Y}$ and $\overline{f}$, concluding the proof.
\end{proof}

\begin{lem}\label{lem_bampmov}
    Let $L$ be an algebraically closed field of characteristic $0$ and $X$ a projective primitive symplectic variety over $L$ with $\Q$-factorial and terminal singularities. We have the equality $\overline{\BAmp(X)} = \clMov(X)$.
\end{lem}
\begin{proof}
    The inclusion $\overline{\BAmp(X)} \subseteq \clMov(X)$ is clear. To prove the converse inclusion, it suffices to show that every rational element in the interior of $\Mov(X)$ is nef on a birational model of $X$.  Take $\alpha \in \Mov^\circ(X) \cap \Pic(X)_\Q$, and let $D$ be a $\Q$-Cartier divisor on $X$ representing the class of $\alpha$. Let $L' \subset L$ be a finitely generated subfield, $X'$ a projective primitive symplectic variety over $\overline{L'} \subset L$ and $D'$ a $\Q$-Cartier divisor on $X'$ such that $X = X'_L$ and $D = D'_L$. Fix an embedding $\overline{L'} \hookrightarrow \C$, and consider the complex projective primitive symplectic variety $X'_\C$. Both $X'$ and $X'_\C$ have $\Q$-factorial terminal singularities by Lemma \ref{lem_qfactterm_basechange}. Recall that we have isometries
    \begin{equation}\label{eq_isomlembamp}
        \Pic(X)_\R \simeq \Pic(X')_\R \simeq \Pic(X'_\C)_\R
    \end{equation}
    induced by pullback of line bundles by Lemma \ref{lem_piciso}. The movable cones of $X$, $X'$ and $X'_\C$ correspond bijectively through the above isometries by Lemma \ref{lem_amplemovbasechange}, therefore $D'$ and $D'_{\C}$ lie respectively in $\Mov^\circ(X')\cap \Pic(X')_\Q$ and $\Mov^\circ(X'_\C) \cap \Pic(X'_\C)_\Q$. Since $\Mov^\circ(X'_\C) \subset \Pos(X'_\C) \subset \Big(X'_\C)$ by \cite[Lemma 4.6, Lemma 4.7]{lehn2024morrison}, the same is true for $X'$ once again thanks to (\ref{eq_isomlembamp}) and Lemma \ref{lem_amplemovbasechange}, thus $D'$ is big. 
    
    We claim that only flips appear in any $K_{X'} + D' = D'$ log-MMP on $X'$. Suppose by contradiction that a divisorial contraction $f$ appears. The base change $f_\C$ remains a divisorial contraction since the relative Picard rank of $f_\C$ remains equal to one by \cite[Proposition 3.1]{davesh2012neron} (note that although the proposition is stated for smooth varieties, the arguments in the proof hold in greater generality). The rest of the argument follows the proof of \cite[Proposition 5.6]{lehn2024morrison}: there exists a $D'_{\C}$-negative curve $R$ on $X'_{\C}$ which can be contracted, and which covers a divisor $E \subset X'_{\C}$. By \cite[Corollary 2.14]{lehn2024morrison}, the curve $R$ deforms along the its Hodge locus and at the general point $t$ of the latter, the Picard number $\rho(X'_{\C,t})$ is one. This means that the class of the deformation $R_t$ of $R$ is dual to a positive multiple of the class of the deformation $E_t$ of $E$, which means that there is $\lambda > 0$ such that $R = \lambda E^\vee$. Then $0 > R \cdot D'_{\C} = \lambda q_{X'_\C}(E,D'_{\C}) \geq 0$, which leads to a contradiction. 
    
    Running an MMP with scaling for a suitable ample divisor
    \cite[Corollary 21.9]{lyu2022relative}, we obtain a birational model $\phi : X' \dashrightarrow Y$ for which $\phi_*D'$ is nef, and that completes the proof.
\end{proof}

\begin{lem}\label{lem_bampext}
    Let $L \subset M$ be a field extension of algebraically closed field of characteristic $0$ and $X$ a projective primitive symplectic variety over $L$ with $\Q$-factorial terminal singularities. Denote by $p : X_M \to X$ the projection. The isomorphism $p^* : \Pic(X)_\R \to \Pic(X_M)_\R$ from Lemma \ref{lem_piciso} induces a bijection between $\BAmp(X)$ and $\BAmp(X_M)$.
\end{lem}
\begin{proof}
    By base change, it is clear that $p^*\BAmp(X) \subseteq \BAmp(X_M)$. For the reverse inclusion, observe that $p^*\clMov(X) = \clMov(X_M)$ by Lemma \ref{lem_amplemovbasechange}. Therefore, Lemma \ref{lem_bampmov} implies that $p^*\BAmp(X)$ is dense in $\Mov(X_M)$. Let $f : X_M \dashrightarrow Y$ be a birational map as in Definition \ref{defn_bamp}. Since $f^*\Amp(Y)$ is an open subset of $\Mov(X_M)$, there exists a birational model $g : X \dashrightarrow Z$ as in Definition \ref{defn_bamp} such that 
    \[p^*g^*\Amp(Z) \cap f^*\Amp(Y) \neq \emptyset.\]
    If $g_M : X_M \dashrightarrow Z_M$ denotes the base change of $g$, we have $p^*g^*\Amp(Z) = g_M^*\Amp(Z_M)$, therefore
    \[f^*\Amp(Y) = g_M^*\Amp(Z_M) \subseteq p^*\BAmp(X).\] 
    This shows that $\BAmp(X_M) \subseteq p^*\BAmp(X)$ and completes the proof. 
\end{proof}

\begin{cor}\label{cor_finmarkmod}
    Let $L$ be an algebraically closed field of characteristic zero and $X$ a projective primitive symplectic variety over $L$ with $b_2(X) \geq 5$ and $\Q$-factorial terminal singularities. Then, $X$ has a finite number of marked minimal models $X'$ (Definition \ref{def_minmodel}) over $L$.
\end{cor}
\begin{proof}
    Take a finitely generated subfield $L' \subset L$ and let $X'$ be a projective primitive symplectic variety over $\overline{L'} \subset L$ such that $X = X'_L$. Fix an embedding $\overline{L'} \hookrightarrow \C$, and consider the complex projective primitive symplectic variety $X'_\C$. By \cite[Corollary 1.4]{lehn2024morrison}, $X'_\C$ has finitely many marked minimal models, therefore $\BAmp(X'_\C)$ is a finite union of pullbacks of ample cones. Moreover, these cones are precisely the connected components of $\BAmp(X'_\C)$. Using Lemma \ref{lem_bampext}, we deduce that $\BAmp(X')$ and $\BAmp(X)$ have finitely many connected components. This shows that the union defining the birational ample cone of $X$ is in fact finite, and it follows that the number of marked minimal models of $X$ is finite.
\end{proof}

\begin{defn}\label{defn_sigma}
    Let $L$ be an algebraically closed field of characteristic $0$, and let $X$ be a projective primitive symplectic variety over $L$ with $b_2(X) \geq 5$ and $\Q$-factorial terminal singularities. Let $\Sigma(X) \subset \Pic(X)_\R$ be the set
    \[
\Sigma(X) \coloneqq
\left\{
f^{\ast} (\alpha) \left| 
\begin{array}{l}
f\colon X \dashrightarrow Y \textup{ is a birational map of } \Q \textup{-factorial terminal projective primitive}\\
\textup{symplectic varieties over } L, \textup{ and } \alpha \in \Nef(Y)^{\vee} \textup{ is integral, primitive, and extremal}
\end{array}
\right.
\right\}.
\]
Here, $\Nef(Y)^\vee$ is the dual cone of $\Nef(Y)$ in $\Pic(Y)_\R$ with respect to the Beauville-Bogomolov-Fujiki form.
\end{defn}
As in the complex case, the square with respect to the Beauville-Bogomolov-Fujiki form of the elements of $\Sigma(X)$ is uniformly bounded from below.

\begin{prop}\label{prop_squarebounded}
    Let $L$ be an algebraically closed field of characteristic $0$, and let $X$ be a projective primitive symplectic variety over $L$ with $b_2(X) \geq 5$ and $\Q$-factorial terminal singularities. Then there exists $B > 0$ such that $q_{X}(x) \geq -B$ for all $x$ in $\Sigma(X)$.
\end{prop}
\begin{proof}
    By Corollary \ref{cor_finmarkmod}, $X$ has a finite number $n \geq 1$ of marked minimal models. For each $i \in \{1, \dotsc, n\}$, let $f_i : X \dashrightarrow Y_i$ be a birational map defining a marking. Take a finitely generated subfield $L' \subset L$ such that $X$, the $Y_i$ and the $f_i$ are defined over $\overline{L'} \subset L$. This means that there exist projective primitive symplectic varieties $X'$ and $Y'_i$ over $\overline{L'} \subset L$, together with birational maps $f'_i : X' \dashrightarrow Y'_i$, such that $X = X'_L$, $Y_i = Y'_{i,L}$ and $f_i = f'_{i,L}$. By Lemma \ref{lem_qfactterm_basechange}, $X'$ and the $Y'_i$ are projective primitive symplectic varieties with $\Q$-factorial terminal singularities. Fix an embedding $\overline{L'} \hookrightarrow \C$ and consider the base changes $X'_{\C}$, $Y'_{i,\C}$ and $f'_{i,\C}$. For the same reason, $X'_{\C}$ and the $Y'_{i,\C}$ also have $\Q$-factorial terminal singularities. By Lemma \ref{lem_piciso}, there are isometries
    \[\Pic(X)_\R \simeq \Pic(X')_\R \simeq \Pic(X'_\C)_\R, ~\text{and}~ \Pic(Y_i)_\R \simeq \Pic(Y'_i)_\R \simeq \Pic(Y'_{i,\C})_\R\]
    for all $i \in \{1, \dotsc, n\}$. For each index $i$, the nef cones of $Y_i$, $Y'_i$ and $Y'_{i,\C}$ correspond bijectively under these isometries by Lemma \ref{lem_amplemovbasechange}. Moreover, since isomorphisms descend along base changes induced by field extensions by \cite[\href{https://stacks.math.columbia.edu/tag/02L4}{Tag 02L4}]{stacks-project}, two distinct marked minimal models of $X'$ cannot become isomorphic after base change to $L$ or to $\C$. Therefore, the projection $X \to X'$ induces a bijection $\Sigma(X) \simeq \Sigma(X')$, and the projection $X'_{\C} \to X'$ induces an injective map $\Sigma(X') \hookrightarrow \Sigma(X'_{\C})$ such that the composite
    \[\Sigma(X) \hookrightarrow \Sigma(X'_{\C})\]
    preserves the Beauville-Bogomolov-Fujiki form. By \cite[Proposition 7.7]{lehn2024morrison}, there exists a real number $B > 0$ such that, for every $x \in \Sigma(X'_{\C})$, we have $q_{X'_{\C}}(x) \geq -B$. In particular, this inequality holds for every element of $\Sigma(X)$, which completes the proof.
\end{proof}
We can now prove Theorem \ref{thm_conjcar0}.(2) by adapting the proof of \cite[Theorem 1.2]{lehn2024morrison}.

\begin{proof}[Proof of Theorem \ref{thm_conjcar0}.(2)]
    Let $\Pi \subset $ be a rational polyhedral cone which is a fundamental domain under the action of $\Bir(X)$ on $\clMov^+(X)$ as in Theorem \ref{thm_conjcar0}.(1). By Proposition \ref{prop_squarebounded} and \cite[Proposition 3.4]{markman2015proof}, the set
    \[\Sigma_{\Pi} := \enstq{\alpha \in \Sigma}{\alpha^{\perp} \cap \Pi \neq \emptyset}\]
    is finite. Cutting out $\Pi$ by the elements of $\Sigma_{\Pi}$ yields a finite decomposition
    \begin{equation}\label{eq_subdivcone}
        \Pi = \cup_{i \in I} \Pi_i
    \end{equation}
    of $\Pi$ into closed rational polyhedral subcones, each with a non-empty interior, with $I$ a finite index set. Let $\Pi_i$ be one of these subcones and $f : X \dashrightarrow Y$ be a birational map as in the definition of $\Sigma$ such that $f^*(\Amp(Y)) \cap g^*(\Pi_i)$ for some $g \in \Bir(X)$. We prove the following equality:
    \begin{equation}\label{eq_pinef}
        g^*(\Pi_i) = f^*(\Nef(Y)) \cap g^*(\Pi).
    \end{equation}
    It suffices to establish it on the interiors. Denote by $\Pi^\circ$ and by $\Pi_i^\circ$ the interiors of $\Pi$ and $\Pi_i$ respectively. The non-empty intersection $\Lambda := g^*(\Pi^0) \cap f^*(\Amp(Y))$ is connected since both $g^*(\Pi^0)$ and $f^*(\Amp(Y))$ are convex cones. As $\Lambda$ cannot intersect $\alpha^{\perp}$ for any $\alpha \in \Sigma_{\Pi}$, we have $\Lambda \subset g^*(\Pi_i^\circ)$. For the reverse inclusion, since the connected components of $\BAmp(X)$ are all of the form $h^*(\Amp(Y'))$ for some birational map $h : X \dashrightarrow Y'$ as in the definition of $\Sigma$, we have $g^*(\Pi_i^\circ) \subset \Lambda$. This proves (\ref{eq_pinef}). 
    
    Consider the set
    \[I_f := \enstq{i \in I}{\exists g_i \in \Bir(X) ~\text{such that}~ g_i^*(\Pi_i) \subset f^*(\Nef(Y))}.\]
    Given that $\Pi$ is a fundamental domain for the $\Bir(X)$-action on $\clMov^+(X)$ and that $f^*(\Nef(Y))$ is contained in it, we see that $I_f$ is non-empty. Moreover, we claim that $I_f$ is uniquely determined by the isomorphism class of $Y$. Indeed, let $f' : X \dashrightarrow Y'$ be another birational map as in the definition of $\Sigma$. If $\phi : Y \to Y'$ is an isomorphism, then $\psi := f'^{-1} \circ \phi \circ f \in \Bir(X)$ satisfies $\psi^*f'^*(\Nef(Y')) = f^*(\Nef(Y))$. Therefore for any $i \in I$ and any $g \in \Bir(X)$, we have $g^*(\Pi_i) \subset f'^*(\Nef(Y'))$ if and only if $(g \circ \psi)^*(\Pi_i) \subset f^*(\Nef(Y))$, thus $i \in I_{f'}$ if and only if $i \in I_{f}$. Conversely, if $I_f = I_{f'}$, take $i \in I_Y = I_{Y'}$ and take $g,g' \in \Bir(X)$ such that $g^*(\Pi_i) \subset f^*(\Nef(Y))$ and $g'^*(\Pi) \subset f'^*(\Nef(Y'))$. Then $f \circ g^{-1} \circ g' \circ f'^{-1} : Y \dashrightarrow Y'$ pulls back an ample class to an ample class, so extends to an isomorphism, and that proves the claim.
    
    Consequently, we set $I_X := I_h$ for any birational map $h \in \Bir(X)$. Let $i \in I_X$ and consider the set $\Bir_{i,X} := \enstq{g \in \Bir(X)}{g^*(\Pi_i) \subset \Nef(X)}$. We claim that $\Bir_{i,X}$ is a left $\Aut(X)$-coset for the natural action of $\Aut(X)$ on $\Bir(X)$. Indeed, assume that for $g,h \in \Bir(X)$, we have $g^*(\Pi_i) \subset \Nef(X)$ and $h^*(\Pi_i) \subset \Nef(X)$. By taking an ample class $\alpha$ in the interior of $\Pi_i$, we see that $g^*(\alpha)$ and $h^*(\alpha)$ are ample lasses, and that $g^{-1} \circ h$ pulls back the ample class $g^*(\alpha)$ the the ample class $h^*(\alpha)$. Therefore, $g^{-1} \circ h$ extends to an isomorphism, and this proves the claim.
    
    Observe that
    \[\Nef^+(X) \subset \clMov^+(X) \subset \bigcup_{g \in \Bir(X)} g^*(\Pi) = \bigcup_{g \in \Bir(X), ~i \in I} g^*(\Pi_i).\]
    Indeed, the first inclusion follows from the fact that $\Amp(X) \subset \BAmp(X) \subset \Mov(X)$, the second inclusion holds by Theorem \ref{thm_conjcar0}.(1) and the last inclusion is just a consequence of (\ref{eq_subdivcone}). We deduce from (\ref{eq_pinef}) that $\Nef^+(X)$ is equal to the union of the translates of $\Pi_i$ by elements of $\Bir(X)$ intersecting its interior, which are in turn union of translates by elements of $\Aut(X)$ since the $\Bir_{i,X}$ are left $\Aut(X)$-cosets for all $i \in I_X$.
    
    Finally, let $G$ be the image of $\Aut(X)$ inside $O(\Pic(X)_{\tf})$ and let $y \in \Amp(X) \cap \Pic(X)_\Q$ be a rational ample class such that its stabilizer in $G$ is trivial. Consider the following set
    \[\cD_{y} := \enstq{x \in \Nef(X)}{\forall g \in G : q_X(g(y),x) \geq q_X(y,x)}.\]
    As $\Nef^+(X)$ is a union of $\Aut(X)$-translates of the finitely many rational polyhedral cones $\Pi_i$ for $i \in I_X$, we deduce that $\cD_{y}$ is a fundamental rational polyhedral domain under the action of $\Aut(X)$ on $\Nef^+(X)$ by \cite[Lemma 7.8]{lehn2024morrison}, and this completes the proof.
\end{proof}

\section{Relative cone conjecture for singular primitive symplectic varieties}
In this section, we will work exclusively over the field $\C$ of the complex numbers. This section is devoted to the proof of Theorem \ref{thm_relconjpsv}. 

\subsection{Preliminaries}\label{sect_relative_prelim}
We recall some classical definitions from birational geometry and the Minimal Model Program using the terminology of \cite[Section 2]{Kollár_Mori_1998}.

\begin{defn}
    Let $f : X \to S$ be a fibration between $\Q$-factorial normal varieties. Two divisors $D$ and $D'$ on $X$ are $f$-linearly equivalent, written $D \sim_{f} D'$ or $D \sim_S D'$, if their difference is linearly equivalent to the pullback of a divisor on $S$. We denote by $\Pic(X/S)$ the relative Picard group (of $f$), and for $\K = \Q$ or $\R$, we set
    \[\Pic(X/S)_{\K} = \Pic(X/S) \otimes_{\Z} \K.\]
\end{defn}

\begin{defn}
    Let $f : X \to S$ be a fibration between $\Q$-factorial normal varieties. A $\Q$-divisor $D$ on $X$ is
    \begin{itemize}
        \item $f$-effective if the restriction of $mD$ on the generic fibre $X_\eta$ is effective for some sufficiently divisible natural number $m$,
        \item $f$-movable if for some sufficiently divisible natural number $m$, we have $\codim(\Supp(\coker \operatorname{ev})) \geq 2$, where $\operatorname{ev}$ is the natural evaluation
        \[\operatorname{ev} : f^*f_*\Ox_X(mD) \to \Ox_X(mD).\]
    \end{itemize}
    An $\R$-divisor $D$ is $f$-effective, respectively $f$-movable, if it is a linear combination with non-negative real coefficients of $f$-effective, respectively $f$-movable, $\Q$-divisors.
\end{defn}

\begin{defn}
    Let $f : X \to S$ be a fibration between $\Q$-factorial normal varieties. Two $\R$-divisors $D$ and $D'$ on $X$ are $f$-numerically equivalent, written $D \equiv_{f} D'$ or $D \equiv_S D'$, if $D \cdot C = D' \cdot C$ for every curve $C \subset X$ such that $f(C)$ is a point. We denote by $N^1(X/S)$ the relative Néron-Severi group (of $f$), and for $\K = \Q$ or $\R$, we set
    \[N^1(X/S)_\K = N^1(X/S) \otimes_{\Z} \K.\]
\end{defn}

\begin{defn}
    Let $f : X \to S$ be a fibration between $\Q$-factorial normal varieties. Inside the vector space $N^1(X/S)_\R$, we define the following convex cones.
    \begin{itemize}
        \item The relative effective cone $\Eff(X/S)$, generated by $f$-effective divisors.
        \item The relative movable cone $\Mov(X/S)$, generated by $f$-movable divisors, and its closure $\clMov(X/S)$. We set $\clMov^e(X/S) := \clMov(X/S) \cap \Eff(X/S)$, and we define $\clMov^+(X/S)$ as the convex hull of $\clMov(X/S) \cap N^1(X/S)_\Q$ in $N^1(X/S)_\R$.
        \item The relative ample cone $\Amp(X/S)$, generated by $f$-ample divisors, and its closure $\Nef(X/S)$, called the relative nef cone. We set $\Nef^e(X/S) := \Nef(X/S) \cap \Eff(X/S)$, and we define $\Nef^+(X/S)$ as the convex hull of $\Nef(X/S) \cap N^1(X/S)_\Q$ in $N^1(X/S)_\R$. 
    \end{itemize}
\end{defn}

\begin{defn}
    A birational map $g : X \dashrightarrow Y$ between normal projective varieties is a contraction if its inverse $g^{-1}$ does not contract a divisor. In addition, if for two fibrations $f_X : X \to S$ and $f_Y : Y \to S$, we have $f_X = f_Y \circ g$ on the locus where $g$ is defined, we say that $g$ is a birational contraction over $S$.
\end{defn}

\begin{defn}
    Let $(X_1, \Delta_1)$ and $(X_2, \Delta_2)$ be two $\Q$-factorial klt pairs. We say that a birational map $\mu : X_1 \dashrightarrow X_2$ is a small $\Q$-factorial modification of $(X_1, \Delta_1)$ over $S$ if $\mu$ is an isomorphism in codimension $1$ over $S$ and if $\mu_*\Delta_1 = \Delta_2$, where $\mu_*\Delta_1$ is the birational transform of $\Delta_1$ by $\mu$.
\end{defn}

\begin{defn}
    Let $f : (X, \Delta) \to S$ be a fibration between normal quasi-projective varieties. A pseudoautomorphism of $(X,\Delta)$ is a birational map $\mu : X \dashrightarrow X$ over $S$ that is an isomorphism in codimension $1$ and that satisfies $\mu_*\Delta = \Delta$. The set of pseudoautomorphisms of $(X,\Delta)$ forms a group denoted by $\psaut(X/S,\Delta)$. When $\Delta = 0$, we simply write $\psaut(X/S)$. The set of automorphisms of $(X,\Delta)$ preserving the boundary $\Delta$ under birational transform will be denoted by $\Aut(X/S,\Delta)$.
\end{defn}

\begin{defn}\label{def_minmodel}
    Let $X \to S$ be a fibration with $X$ a normal projective variety. Let $\Delta$ be a divisor on $X$ such that $(X,\Delta)$ is a log-canonical pair. For every birational contraction $\phi : X \dashrightarrow Y$ of normal projective varieties over $S$, we set $\Delta_Y := \phi_* \Delta$. A \emph{minimal model} $(Y,\Delta_Y)$ of $(X,\Delta)$ over $S$ is a pair associated to a birational contraction $\phi : (X,\Delta) \dashrightarrow (Y,\Delta_Y)$ such that
    \begin{itemize}
        \item $Y$ is $\Q$-factorial,
        \item $K_Y + \Delta_Y$ is nef over $S$ (where $K_Y$ is the canonical divisor of $Y$), and
        \item $a(E,X,\Delta) > a(E,Y,\Delta_Y)$ for all $\phi$-exceptional divisors.
    \end{itemize}
    We will say that a minimal model $(Y,\Delta_Y)$ over $S$ is \emph{a good minimal model} if, in addition, $K_Y + \Delta_Y$ is semiample over $S$. A minimal model $(Y, \Delta_Y)$ together with the birational contraction $\phi$ is called a \emph{marked minimal model} of $(X, \Delta)$ over $S$.
\end{defn}
In the next section, we will also need to refer to the weak cone conjecture; we therefore take the opportunity in this preliminary section to introduce its definition following \cite{LOP2018ConeConjecture} and \cite{li2023relative}.

\begin{defn}
    Let $V$ be a real vector space, $C \subset V$ a cone and $\rho : \Gamma \hookrightarrow \GL(V)$ an injective group homomorphism. Let $\Pi \subset C$ be a (rational) polyhedral cone. Suppose that $\Gamma$ acts on $C$. Then $\Pi$ is called a weak (rational) polyhedral fundamental domain for $C$ under the action of $\Gamma$ if
    \begin{itemize}
        \item $\Gamma \cdot \Pi = C$, and
        \item for each $\gamma \in \Gamma$, either $\gamma \Pi = \Pi$ or $\gamma \Pi \cap \Pi^\circ = \emptyset$, where $\Pi^\circ$ is the interior of $\Pi$.
    \end{itemize}
    Moreover, let $\Gamma_{\Pi} := \enstq{\gamma \in \Gamma}{\gamma \Pi = \Pi}$. If $\Gamma_{\Pi} = \{\id\}$, then $\Pi$ is called a (rational) polyhedral fundamental domain.
\end{defn}

\begin{conj}[Weak Kawamata-Morrison cone conjectures]\label{conj_weak}
    Let $(X,\Delta) \to S$ be a fiber space with trivial canonical class as in Definition \ref{def_ktrivfibrespace}.
    \begin{itemize}
        \item $\clMov^+(X/S)$ has a weak rational polyhedral fundamental domain under the action of $\psaut(X/S,\Delta)$.
        \item $\Nef(X/S)^+$ has a weak rational polyhedral fundamental domain under the action of $\Aut(X/S,\Delta)$.
    \end{itemize}
\end{conj}

\subsection{Proof of the relative cone conjecture}\label{sect_relconeconj}
We first recall the following lemma which relies on \cite[Lemma 3.7]{li2023relative} and the proof of \cite[Theorem 1.3]{li2023relative}. 

\begin{lem}[Lemma 5.1 - \cite{horing2024relative}]\label{lem_surjconens}
    Let $X$ be a $\Q$-factorial variety, and let $f : X \to S$ be a fibration. Denote by $\eta \in S$ the generic point of the base, and by $X_\eta$ the generic fiber of $f$. We have a surjective map
    \[r : N^1(X/S)_\R \to N^1(X_\eta)_\R\]
    which induces the surjective maps
    \[\Eff(X/S) \to \Eff(X_\eta) ~\text{and}~ \clMov(X/S) \to \clMov(X_\eta)\]
\end{lem}

\begin{rmk}
    Let $f : (X,\Delta) \to S$ be a $K$-trivial fiber space $f : (X,\Delta) \to S$ (cf. Definition \ref{def_ktrivfibrespace}) whose very general fibre $X_s$ has a vanishing irregularity $q(X_s)$, and assume the base $S$ is $\Q$-factorial. As it was observed in \cite[Remark 1.6]{horing2024relative} and in the remark following \cite[Lemma 5.1]{horing2024relative}, the following properties hold:
    \begin{itemize}
        \item the cone $\clMov(X/S)$ is non-degenerate;
        \item $R^1 f_*\Ox_X = 0$;
        \item $N^1(X/S)_\Q \simeq \Pic(X/S)_\Q$;
        \item $N^1(X_\eta)_\Q \simeq \Pic(X_\eta)_\Q$.
    \end{itemize}
\end{rmk}

\begin{defn}
    Let $f : X \to S$ be a fibration between complex quasi-projective varieties. If $\eta \in S$ denotes the generic point of $S$, we denote by $\overline{\eta} \in S$ the image of the composition $\Spec \overline{\C(S)} \to \Spec \C(S) \to S$, where $\C(S)$ is the function field of $S$. We call $\overline{\eta}$ the \emph{geometric generic point} of $S$, and $X_{\overline{\eta}}$ the \emph{geometric generic fibre} of $f$.
\end{defn}

\begin{lem}\label{lem_genfibrepsv}
    Let $X$ and $S$ be complex quasi-projective varieties, and let $f : X \to S$ be a projective fibration. Denote by $\eta \in S$ the generic point of $S$ and set $F := \C(S)$. If the very general fibre $X_s$ of $f$ is a complex projective primitive symplectic variety with $\Q$-factorial terminal singularities and $b_2(X_s) \geq 5$, then the generic fibre $X_{\eta}$ (respectively, the geometric generic fibre $X_{\overline{\eta}}$) is a projective primitive symplectic variety over $F$ (respectively, over $\overline{F}$) with $\Q$-factorial terminal singularities and $b_2(X_{\eta}) \geq 5$.
\end{lem}
\begin{proof}
    By Lemma \ref{lem_descentofpsvfg}, in order to show that $X_{\eta}$ is a projective primitive symplectic variety over $F$, it suffices to prove that $X_{\overline{\eta}}$ is a primitive symplectic variety over $\overline{F}$. One way to proceed is to exploit the fact that we are working over $\C$ and to apply \cite[Lemma 2.1]{vial2013algebraic}. We can then find a closed point $s \in S$ such that $X_s$ is a complex primitive symplectic as in the statement of this lemma, an (abstract) field isomorphism $\phi : \C \to \overline{F}$, and an isomorphism of schemes $\alpha : X_{\overline{\eta}} \to X_s$ over $\Spec(\phi)$. It follows that
    \[X_{\overline{\eta}} \simeq X_s \otimes_{\C,\phi} \overline{F}\]
    is a projective primitive symplectic variety over $\overline{F}$ by Lemma \ref{lem_extofPSV}, and $b_2(X_{\eta}) \geq 5$. Moreover, the isomorphism $\alpha$ induces isomorphisms between the divisor class groups and the Picard groups of $X_s$ and $X_{\overline{\eta}}$, and the following diagram is commutative.
    \begin{center}
        \begin{tikzcd}
\Pic(X_s) \otimes \Q \arrow[d] \arrow[rr, "\simeq"]    &  & \Pic(X_{\overline{\eta}}) \otimes \Q \arrow[d]    \\
\operatorname{Cl}(X_s) \otimes \Q \arrow[rr, "\simeq"] &  & \operatorname{Cl}(X_{\overline{\eta}}) \otimes \Q
\end{tikzcd}
    \end{center}
    It follows that $X_s$ is $\Q$-factorial if and only if $X_{\overline{\eta}}$ is $\Q$-factorial. Since $K_{X_{\overline{\eta}}} = \alpha^*K_{X_s}$ we also see that the singularities of $X_s$ are terminal if and only if those of $X_{\overline{\eta}}$ are also terminal. Therefore, by \cite[Lemma 6.18]{das2022lmmp} and \cite[Proposition 2.15]{kollar2013singularities}, the singularities of $X_{\eta}$ are $\Q$-factorial and terminal, and that completes the proof.
\end{proof}
We now state a weaker version of \cite[Lemma 2.17]{denisi2022pseudo} in the singular setting, which will be useful in what follows.

\begin{lem}\label{lem_nefsemiample}
	Let $X$ be a complex $\Q$-factorial projective primitive symplectic variety with terminal singularities. Suppose that for every birational map $\phi : X \dashrightarrow X'$ with $X'$ another $\Q$-factorial projective primitive symplectic variety whose singularities are terminal, every nef divisor on $X'$ is semi-ample. Then $\clMov^+(X) \subseteq \clMov^e(X)$.
\end{lem}
\begin{proof}
Since the interiors of the cones are the same, we only have to check that the boundary of the left-hand cone is included in the right-hand cone. Let $\alpha \in \partial \clMov^+(X)$ and, without loss of generality, suppose that $\alpha$ is integral. Since $\overline{\BAmp(X)} = \clMov(X)$ by \cite[Proposition 5.8]{lehn2024morrison}, there exists a projective primitive symplectic variety $X'$ with $\Q$-factorial and terminal singularities, and a birational map $\phi : X \dashrightarrow X'$ such that $\phi_*\alpha$ is integral and nef. Thanks to our assumptions, $\phi_*\alpha$ is semi-ample on $X'$, and is therefore effective. Since $\phi$ is an isomorphism in codimension one, $\alpha$ is effective and this concludes the proof.
\end{proof}
The following result is an adaptation of \cite[Proposition 5.5]{horing2024relative} to the singular setting.

\begin{prop}\label{prop_nefmovrel}
	Let $f : X \to S$ be a projective fibration between $\Q$-factorial varieties, and let $X_s$ denote a very general fibre of $f$. Let $\eta \in S$ be the generic point of $S$, and denote by $j : X_\eta \to X$ the inclusion. Then the following assertions hold.
	\begin{enumerate}
	\item If $\Nef^+(X_s) \subseteq \Nef^e(X_s)$, then $\Nef^+(X_\eta) \subseteq \Nef^e(X_\eta)$.
	\item Assume that $X_s$ is a complex projective primitive symplectic variety with $\Q$-factorial terminal singularities. If the good minimal models exist for effective klt pairs on $X_s$, then 
	\[\clMov^e(X_\eta) = \clMov^+(X_\eta) \text{ and } \Mov^+(X/S) = \Mov^e(X/S).\]
	\end{enumerate}
\end{prop}
\begin{proof}
(1) Let $D_\eta \in \Nef^+(X_\eta)$ and, without loss of generality, assume that $D_\eta$ is rational. By Lemma \ref{lem_surjconens}, we may choose $D \in N^1(X)_\Q$ such that $D \vert_{X_\eta} = D_\eta$. Let $H$ be an $f$-ample divisor on $X$. Observe that, for any integer $m > 0$, the divisor
\begin{equation}\label{eq_divample}
\left(D + \dfrac{1}{m}H \right)\vert_{X_\eta}
\end{equation}
is ample on $X_\eta$. Moreover, by \cite[Corollaire 9.6.4]{grothendieck1966elements}, for each such integer $m$, there exists an open subset $U_m \subseteq S$ such that $\left(D + m^{-1}H \right)\vert_{X_s}$ is ample on $X_s$ for all $s \in U_m$. By (\ref{eq_divample}), these open subsets are non-empty. Set $U := \cap_{m \geq 1} U_m$. By construction, $D \vert_{X_s}$ is nef for all $s \in U$, and our hypothesis implies that $D \vert_{X_s}$ is also effective for all $s \in U$. Hence $D \in \Eff(X/S)$, and therefore $D_\eta \in \Eff(X_\eta)$.

(2) We first prove the inclusion $\clMov^e(X_\eta) \subset \clMov^+(X_\eta)$. Let $D_\eta \in \clMov^e(X_\eta)$. By Lemma \ref{lem_surjconens}, there exists $D \in \clMov^e(X/S)$ such that $D \vert_{X_\eta} = D_\eta$. By \cite[Lemma 3.4]{horing2024relative}, we have $D \in \clMov^+(X/S)$, thus $D_\eta \in \clMov^+(X_\eta)$. For the reverse inclusion $\clMov^+(X_\eta) \subset \clMov^e(X_\eta)$, let $D_\eta \in \clMov^+(X_\eta)$ and assume, without loss of generality, that $D_\eta$ is rational. By Lemma \ref{lem_surjconens}, we may choose $D \in N^1(X/S)_\Q$ such that $D\vert_{X_\eta} = D_\eta$. Let $H$ be an $f$-ample divisor on $X$. For any integer $m > 0$, the divisor
\begin{equation}\label{eq_divbig}
\left( D + \dfrac{1}{m}H \right)\vert_{X_\eta}
\end{equation}
is big on $X_\eta$, thus we may write $(D + m^{-1}H)\vert_{X_\eta} = A_{m,\eta} + E_{m,\eta}$ with $A_{m,\eta} \in \Amp(X_\eta)$ and $E_{m,\eta} \in \Eff(X_\eta)$. Choose rational classes $A_m \in N^1(X/S)_\Q$ and $E_m \in \Eff(X/S)$ such that $(A_m)\vert_{X_\eta} = A_{m,\eta}$ and $(E_m)\vert_{X_\eta} = E_{m,\eta}$. As above, for each integer $m > 0$, there exists a non-empty open subset $U_m \subseteq S$ such that $(D+m^{-1}H)\vert_{X_s} = (A_m + E_m)\vert_{X_s}$ is big for all $s \in U_m$. Setting again $U := \cap_{m \geq 1} U_m$, it follows that $D\vert_{X_s}$ is pseudoeffective for all $s \in U$. By \cite[Theorem 3.2]{lehn2024footnotes}, we may write $D\vert_{X_s} = P_s + N_s$, with $P_s \in \clMov^+(X_s)$ and $N_s \in \Eff(X_s)$. Since $\clMov^+(X_s) \subseteq \clMov^e(X_s)$ by Lemma \ref{lem_nefsemiample}, it follows that $D\vert_{X_s}$ is effective. Therefore, $D$ is $f$-effective, and hence $D_\eta$ is effective.

As for the equality $\clMov^+(X/S)=\clMov^e(X/S)$, the inclusion $\clMov^+(X/S) \subset \clMov^e(X/S)$ can be proved in the same way as above. The reverse inclusion follows from \cite[Lemma 3.4]{horing2024relative}, which concludes the proof.
\end{proof}
We can now prove our final theorem.

\begin{proof}[Proof of Theorem \ref{thm_relconjpsv}]
By Lemma \ref{lem_genfibrepsv}, the generic fibre $X_\eta$ is a projective primitive symplectic variety over $F = \C(S)$ with $b_2(X_\eta) \geq 5$ such that $X_{\overline{\eta}}$ has $\Q$-factorial terminal singularities. Theorem \ref{thm_conjcar0}.(2) implies that $\clMov^+(X_\eta)$ admits a rational polyhedral fundamental domain under the action of $\Bir(X_\eta)$. Under our assumptions, we have $\clMov^+(X_\eta) = \clMov^e(X_\eta)$ by Proposition \ref{prop_nefmovrel}. As observed in \cite{horing2024relative}, the proof of \cite[Theorem 6.1]{li2023relative} shows that the weak cone conjecture \ref{conj_weak} for $\clMov^+(X_\eta)$ implies the weak cone conjecture for $\clMov^+(X/S)$. Finally, since we also have $\clMov^e(X/S) = \clMov^+(X/S)$ by Proposition \ref{prop_nefmovrel}, it follows from \cite[Lemma 2.18]{horing2024relative} that the relative movable cone conjecture holds for $f : X \to S$. Finally, by \cite[Proposition 4.3]{horing2024relative}, there are only finitely many small $\Q$-factorial modifications of $X$ over $S$ up to isomorphism over the base, and the relative nef cone conjecture for each of them. This completes this proof.
\end{proof}

\printbibliography
\Addresses

\end{document}

%% file: macro.tex
\usepackage[utf8]{inputenc}
\usepackage{amsmath,calligra,mathrsfs}
\usepackage[margin=2.5cm]{geometry}
\usepackage[english]{babel}
\usepackage{csquotes}
\usepackage[utf8]{inputenc} 
\usepackage[T1]{fontenc} 
\usepackage{amssymb,amsthm,mathtools,dsfont}
\usepackage{hyperref}
\usepackage{xcolor}
\hypersetup{
    colorlinks = true,
    linkcolor={red!80},
    citecolor={green!60!black},
    urlcolor={blue!80!black},
    pdfauthor=author
}
\usepackage[final]{microtype}
\usepackage{comment,mdframed,multicol}
\usepackage{array}
\usepackage{stmaryrd}
\usepackage{caption}
\usepackage{subcaption}
\usepackage[pdftex]{graphicx}
\usepackage{tipa}
\usepackage{adjustbox}
\usepackage{mathdots}
\usepackage{leftindex}
\usepackage{braket}
\usepackage[shortlabels]{enumitem}
\numberwithin{equation}{section}
\everymath{\displaystyle}

\theoremstyle{definition}
\newtheorem{defn}{Definition}[section]

\theoremstyle{plain}
\newtheorem{thm}[defn]{Theorem}
\newtheorem{lem}[defn]{Lemma}
\newtheorem{cor}[defn]{Corollary}
\newtheorem{prop}[defn]{Proposition}
\newtheorem{defprop}[defn]{Definition/Proposition}
\newtheorem{conj}[defn]{Conjecture}
\newtheorem*{conj*}{Conjecture}
\newtheorem{thmx}{Theorem}

\theoremstyle{remark}
\newtheorem{rmk}[defn]{Remark}

\newcommand{\Z}{\mathbb Z}
\newcommand{\Q}{\mathbb Q}
\newcommand{\R}{\mathbb R}
\newcommand{\C}{\mathbb C}
\newcommand{\K}{\mathbb K}
\newcommand{\sym}{\mathfrak S}
\newcommand{\G}{\mathds G}
\newcommand{\Ox}{\mathcal O}
\renewcommand{\P}{\mathbb P}
\def\GL{\mathop{\rm GL}\nolimits}
\def\PGL{\mathop{\rm PGL}\nolimits}
\def\Hilb{\mathop{\rm Hilb}\nolimits}

\def\cC{{\mathcal C}}
\def\cD{{\mathcal D}}

\def\cF{{\mathcal F}}

\def\cP{{\mathcal P}}

\def\cT{{\mathcal T}}

\def\Eff{\mathop{\rm Eff}\nolimits}
\def\Amp{\mathop{\rm Amp}\nolimits}
\def\Big{\mathop{\rm Big}\nolimits}
\def\Nef{\mathop{\rm Nef}\nolimits}
\def\Mov{\mathop{\rm Mov}\nolimits}
\def\Bir{\mathop{\rm Bir}\nolimits}
\def\bBir{\mathop{\rm \textbf{Bir}}\nolimits}
\def\Pic{\mathop{\rm Pic}\nolimits}
\def\Pos{\mathop{\rm Pos}\nolimits}
\def\BAmp{\mathop{\rm BAmp}\nolimits}
\def\bPic{\mathop{\rm \textbf{Pic}}\nolimits}
\def\Gal{\mathop{\rm Gal}\nolimits}
\def\Brauer{\mathop{\rm Br}\nolimits}

\def\clMov{\mathop{\overline{\rm Mov}}\nolimits}
\def\clPos{\mathop{\overline{\rm Pos}}\nolimits}

\def\Aut{\mathop{\rm Aut}\nolimits}
\def\Spec{\mathop{\rm Spec}\nolimits}
\def\reg{\mathop{\rm reg}\nolimits}
\def\etale{\mathop{\rm \acute{e}t}\nolimits}
\def\tf{\mathop{\rm tf}\nolimits}

\def\monlthdg{\mathop{\rm Mon}^{\mathop{\rm 2,lt}}_{\mathop{\rm Hdg}}\nolimits}
\def\monltbir{\mathop{\rm Mon}^{\mathop{\rm 2,lt}}_{\mathop{\rm Bir}}\nolimits}
\def\psaut{\mathop{\rm PsAut}\nolimits}
\def\bsloc{\mathop{\rm Bs}\nolimits}
\def\Supp{\mathop{\rm Supp}\nolimits}

\newcommand{\id}{\text{id}}

\newcommand{\coker}{\text{coker} ~}

\newcommand{\codim}{\operatorname{codim}}

\newcommand{\abs}[1]{\left\lvert#1\right\rvert}
\newcommand{\enstq}[2]{\left\{#1~\middle|~#2\right\}}

\usepackage[all]{xy} % diagramme commutatif
\usepackage{tikz-cd}

\usepackage[
maxbibnames=99,
maxalphanames=4,
backend=biber,
sorting=nyt,
giveninits=true,
useprefix=true
]{biblatex}

\newcommand{\Addresses}{{% additional braces for segregating \footnotesize
  \bigskip
  \footnotesize

  ~~\textsc{Institut Élie Cartan de Lorraine, Université de Lorraine et CNRS, F-54000 Nancy, France.}\par\nopagebreak
  ~~~~\textit{E-mail address}: \texttt{aurelien.faucher@univ-lorraine.fr}

}}